\documentclass[11pt]{article}
\usepackage[margin=1in]{geometry}
\usepackage{amsmath,amsfonts,amssymb,amsthm}
\usepackage{graphicx}
\usepackage{enumerate}
\usepackage{bbm}
\usepackage{esint}
\usepackage{verbatim}
\usepackage{hyperref,color}
\usepackage[capitalize,nameinlink]{cleveref}
\usepackage[dvipsnames]{xcolor}
\hypersetup{
	colorlinks=true,
	pdfpagemode=UseNone,
    citecolor=OliveGreen,
    linkcolor=NavyBlue,
    urlcolor=Magenta,
	pdfstartview=FitW
}

\usepackage{appendix}
\crefname{appsec}{Appendix}{Appendices}
\usepackage{tikz}
\usepackage[colorinlistoftodos]{todonotes}
\usepackage{fourier}
\usepackage{caption}
\usepackage{subcaption}

\theoremstyle{plain}
\newtheorem{thm}{Theorem}[section]

\newtheorem{lemma}[thm]{Lemma}

\newtheorem{cor}[thm]{Corollary}

\newtheorem{qn}[thm]{Question}

\theoremstyle{definition}
\newtheorem{defn}[thm]{Definition}

\newtheorem{obs}[thm]{Observation}

\newtheorem*{ass*}{Assumption}
\newtheorem*{notn}{Notation}
\newtheorem*{note*}{Note}

\theoremstyle{remark}
\newtheorem{rmk}[thm]{Remark}
\newtheorem{rem}[thm]{Remark}

\crefname{lem}{Lemma}{Lemmas}
\crefname{thm}{Theorem}{Theorems}
\crefname{defn}{Definition}{Definitions}
\crefname{fact}{Fact}{Facts}
\crefname{clm}{Claim}{Claims}
\crefname{prop}{Proposition}{Propositions}

\newcommand{\sd}{1.96}

\DeclareMathOperator{\Area}{Area}

\newcommand{\eps}{\varepsilon}

\newcommand{\Z}{\mathbb{Z}}

\newcommand{\NN}{\mathbb{N}}

\newcommand{\R}{\mathbb{R}}
\newcommand{\RR}{\mathbb{R}}
\newcommand{\TT}{\mathbb{T}}
\newcommand{\EE}{\mathbb{E}}

\newcommand{\mcF}{\mathcal{F}}
\newcommand{\mcG}{\mathcal{G}}

\newcommand{\mcI}{\mathcal{I}}

\newcommand{\ind}{\mathbf{1}}

\newcommand{\mc}{\mathcal}
\newcommand{\ol}{\overline}

\renewcommand{\Pr}{\mathbb{P}}
\renewcommand{\P}{\mathbb{P}}

\title{Clustering in typical unit-distance avoiding sets}
\author{
    Alex Cohen\footnote{MIT; alexcoh@mit.edu}\quad and Nitya Mani\footnote{MIT; nmani@mit.edu}
}

\begin{document}
\maketitle

\begin{abstract}
In the 1960s Moser asked how dense a subset of $\R^d$ can be if no pairs of points in the subset are exactly distance 1 apart.
There has been a long line of work showing upper bounds on this density. One curious feature of dense unit distance avoiding sets is that they appear to be ``clumpy,'' i.e. forbidding unit distances comes hand in hand with having more than the expected number distance $\approx 2$ pairs. 

In this work we rigorously establish this phenomenon in $\RR^2$. We show that dense unit distance avoiding sets have over-represented distance $\approx 2$ pairs, and that this clustering extends to typical unit distance avoiding sets. To do so, we build off of the linear programming approach used previously to prove upper bounds on the density of unit distance avoiding sets.
\end{abstract}

\section{Introduction}

In the early 1960s, Moser~\cite{CRO67} posed the following question: how dense can a subset $A \subset \RR^d$ be if $A$ contains no pair of points that are distance exactly $1$ apart, i.e. if $A$ is \textit{1-avoiding}? This problem is closely related to the Hadwiger-Nelson problem which asks to compute $\chi(\RR^d)$, the minimum cardinality of a coloring of $\RR^d$ so that points distance exactly $1$ apart receive different colors. 

The plane $(d = 2)$ has received much attention, with Erd\H{o}s~\cite{ER85} famously conjecturing in 1985 an upper bound of $1/4$. This conjecture was recently resolved in the affirmative by Ambrus, Csisz\'arik, Matolcsi, Varga, and Zs\'amboki~\cite{ACMVZ22}.

\begin{thm}[\cite{ACMVZ22}]
Any Lebesgue measurable $1$-avoiding planar set has upper density at most $0.2470,$ i.e. upper density strictly less than $1/4$.
\end{thm}

With respect to the $\ell_1$ or $\ell_{\infty}$ norms, $\frac14$ is a tight upper bound for $1$-avoiding sets, as shown in~\cite{KMOFR16}. The curvature of the $\ell_2$ norm forces $1$-avoiding sets to be smaller (see~\cref{r:curvature} for further discussion). Prior to~\cite{ACMVZ22}, one of the strongest pieces of evidence in favor of the Erd\H{o}s conjecture was the work of~\cite{KMOFR16}, who showed that any 1-avoiding set in $\RR^d$ (for $d \ge 2$) that displays \textit{block structure} has density strictly less than $2^{-d}$. A $1$-avoiding set has block structure if the set is a union of blocks where the distance between any pair of points from the same block is less than 1 and the distance between two points from different blocks is greater than $1$. 

 The best known constructions of dense $1$-avoiding sets have block structure. A simple, relatively large $1$-avoiding set can be constructed by placing open circular discs of radius $1/2$ at the lattice points of a regular hexagonal lattice generated by two length $2$ vectors with angle $\pi/3$. The union of these open circular discs is a $1$-avoiding set that has block structure and density $\pi/(8\sqrt{3}) \approx 0.2267$. This construction was strengthened by Croft~\cite{CRO67} as follows: we place a series of \textit{tortoises} (intersections of an open disc of radius $1/2$ with height $x < 1$ open regular hexagons) on a hexagonal lattice generated by two length $1 + x$ vectors with angle $\pi/3$. The maximum density of the union of these tortoises is $\approx 0.22936$, which occurs for $x \approx 0.96553$. This $1$-avoiding set has block structure and is the densest known construction of a $1$-avoiding set in $\RR^2$ (see~\cref{fig:croft}).

\begin{figure}
\centering
\begin{subfigure}{.4\linewidth}
  \hspace{-10pt}
  \includegraphics[width=1\linewidth]{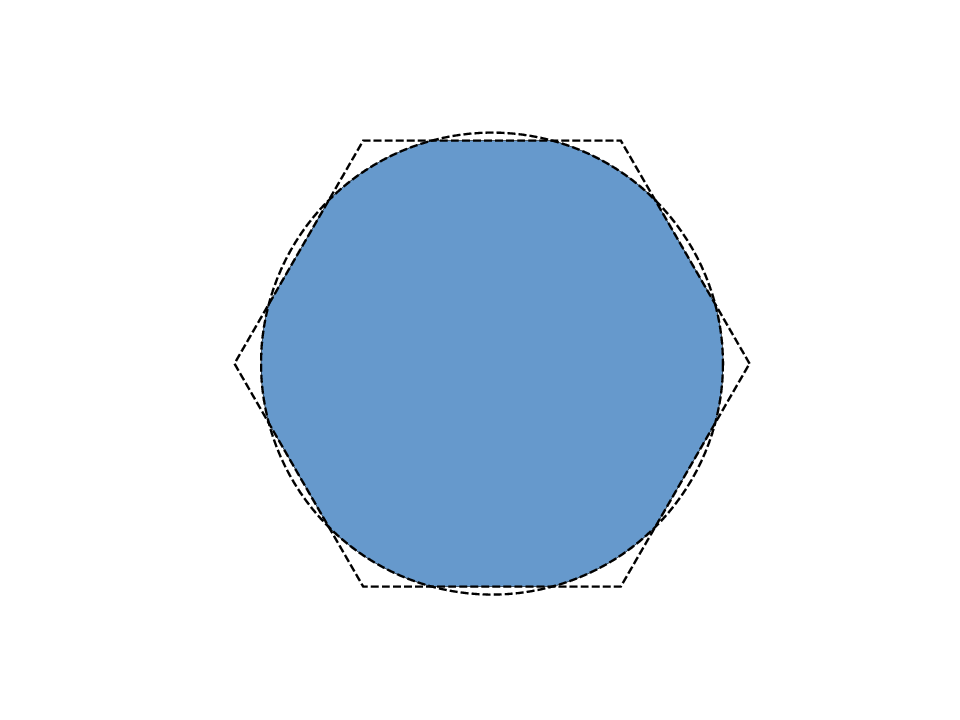}
\end{subfigure}
\begin{subfigure}{.4\linewidth}
  \hspace{-50pt}
  \includegraphics[width=1\linewidth]{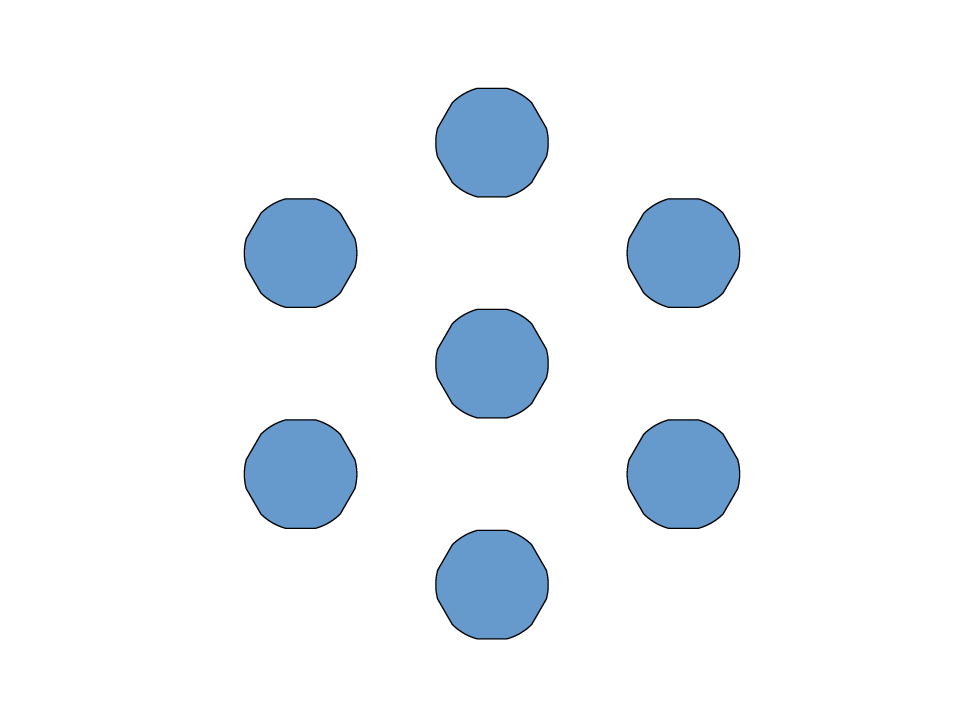}
\end{subfigure}
\caption{Croft's tortoise construction~\cite{CRO67,ACMVZ22}}
\label{fig:croft}
\end{figure}

This clustering is not an artifact of carefully crafted $1$-avoiding sets, but rather a naturally occurring phenomenon. 

As further evidence, we look to the arguments used to establish \textit{upper bounds} on the density of a $1$-avoiding set in $\RR^2$ (including~\cite{SZE84,FV10,KMOFR16,AM22,ACMVZ22}). To describe this approach at a high level, we make a few notions from above precise.

\begin{defn}
For a measurable subset $A \subset \RR^d$, the \textit{upper density} of $A$ is defined as $$\overline{\delta(A)} := \lim \sup_{K \to \infty} \frac{\lambda_d(A \cap [-K, K]^d)}{\lambda_d([-K, K]^d},$$ where $\lambda_d$ is the $d$-dimensional Lebesgue measure (we can define the \textit{lower density} analogously). If the limit of the above quantity exists, we call it the \textit{density} of $A$, denoted $\delta (A)$.

We let $m_1(\RR^d)$ denote the extremal bound on the density of $1$-avoiding sets, i.e. 
$$m_1(\RR^d) = \sup\left\{ \overline{\delta(A)} \mid A \subset \RR^d \text{ is 1-avoiding and measurable}  \right\}.$$
\end{defn}
Much of the recent work improving upper bounds on $m_1(\RR^2)$ has come from a linear programming approach (this procedure and its recent applications are described in more detail in~\cref{s:lp}). At a high level, the LP approach proves upper bounds on $m_1(\RR^2)$ by studying the \textit{autocorrelation function} $f(x) := \delta(A \cap (A - x))$ of a $1$-avoiding set $A$.
By averaging $f$ over the unit circle and normalizing, one gets a radialized function $f^{\circ}(r) := f^{\circ}(r; A)$. Intuitively, $f^{\circ}(r)$ is the \textit{pair correlation function}, it describes the density of distance $r$ pairs in the set $A$. The maximum value is $f^{\circ}(0) = \delta(A)$, and the typical value for a random set is $f^{\circ}(r) \sim \delta(A)^2$. Another way to say that $A$ is $1$-avoiding is to compute that $f^{\circ}(1) = 0$. We can describe $f^{\circ}$ by its Fourier coefficients $\{\kappa(t)\}_{t \in \RR_+}$. Because $f^{\circ}$ is the pair correlation function of a set, the coefficients satisfy several linear constraints. We upper bound $m_1(\RR^2)$ by solving a linear program: maximize $\delta(A)$ subject to linear constraints on the Fourier coefficients $\{\kappa(t)\}_{t \in \RR_+}$. By solving a discretized version of the dual LP, we can \textit{certify} an upper bound on $m_1(\RR^2)$. Solutions to the primal problem give candidate pair correlation functions. Solutions to the dual problem witness constraints on the density coming from constraints on the pair correlation function.

The pair correlation function $f^{\circ}(r; A)$ is an important object in upper-bound proofs, but is also interesting to study in its own right. See~\cref{fig:pair_cor_graph} for the pair correlation function of the hexagonal disk packing (we denote the set corresponding to this packing $A_{\mathrm{disk}}$).
\begin{figure}
    \centering
    \includegraphics[width=0.6\linewidth]{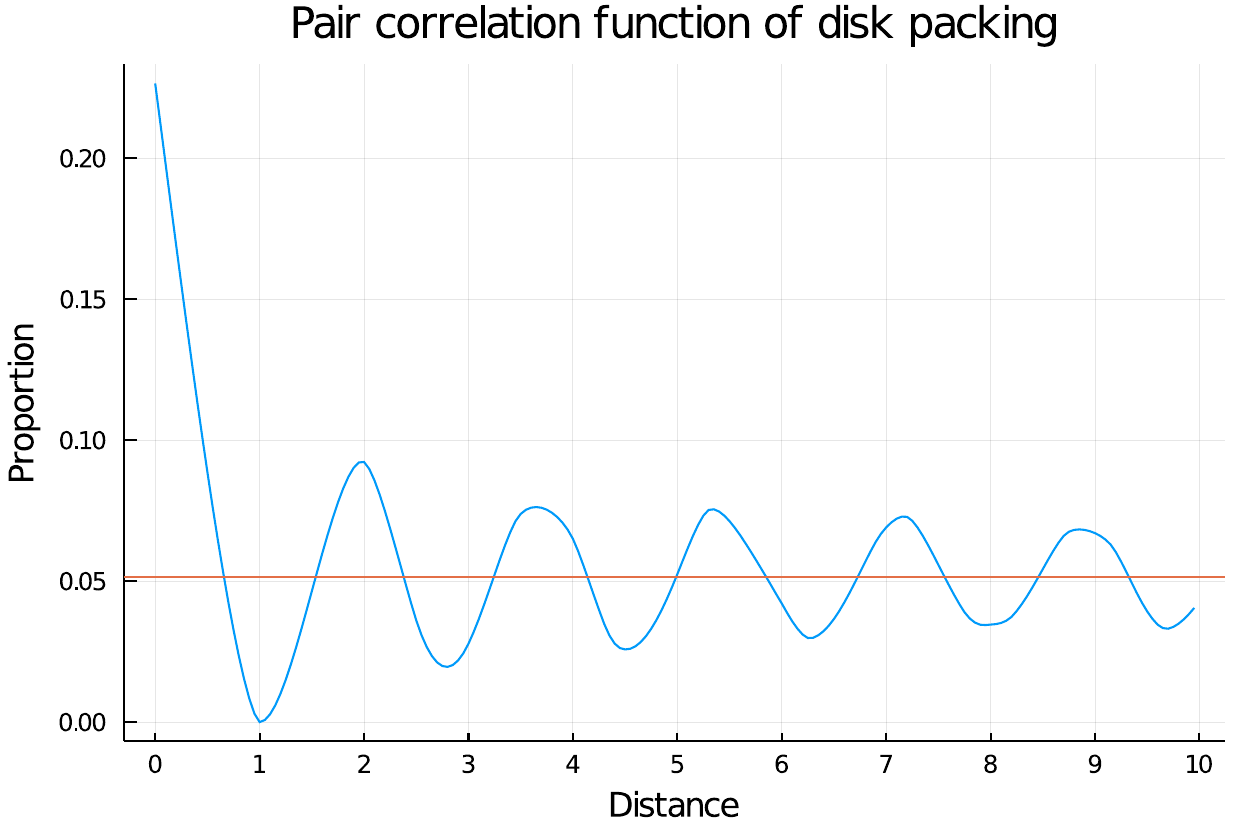}
    \caption{Pair correlation function of the unit distance avoiding set obtained by packing radius $\frac{1}{2}$ disks at distance $1$ from each other. The orange line is the typical pair correlation for a generic set of the same density.}
    \label{fig:pair_cor_graph}
\end{figure}
Let's start by making a few observations about properties of $f^{\circ}$ for a generic set $A$ and the set $A_{\mathrm{disk}}$:
\begin{itemize}
    \item For any $1$-avoiding set $A$ we have $f^{\circ}(1; A) = 0$.
    \item  $f^{\circ}(0; A) = \delta(A)$, the density of the set $A$.
    \item For a randomly generated set of density $\delta$, we expect $f^{\circ}(r) \sim \delta^2$ for most $r$. On the other hand, notice $f^{\circ}(r; A_{\mathrm{disk}})$ is close to $\delta(A)$ for $r$ near zero. This indicates that the set $A_{\mathrm{disk}}$ is ``clumpy'', there are lots of nearby pairs. 
    \item We have $f^{\circ}(2; A_{\mathrm{disk}}) \sim 0.09$, which is more than the typical value $\delta^2 \sim 0.05$. The reason is that $A_{\mathrm{disk}}$ is built out of blocks that have distance $\sim 2$ from each other, so there are extra distance 2 pairs in the whole set. 
\end{itemize} 
In this work we rigorously establish the overrepresentation of distance $\approx 2$ pairs for dense $1$-avoiding sets. Let 
\begin{equation}
    s(r; A) = \frac{f^{\circ}(r; A)}{\delta(A)^2},
\end{equation}
so that for a uniformly random set $A$ we expect $s(r; A) \sim 1$.
\begin{thm}\label{t:main} 
There exists $\gamma > 0$ such that the following holds. Suppose $A \subset \RR^2$ is a periodic set with $s(1; A) \leq \gamma$ and $\delta(A) \geq m_1(\RR^2) - \gamma$. Then 
\begin{equation*}
    s(\sd; A) \geq 1+\gamma. 
\end{equation*}
\end{thm}
By periodic, we mean periodic with respect to the lattice $K\Z^2$ for some $K > 0$. 
It is technically convenient to assume $A$ is periodic because the density and pair correlation function are well defined. In practice any measurable subset of $\R^2$ can be approximated by periodic subsets, so we don't lose generality.

In the above result, we use the value \sd\, instead of 2 for technical reasons, but it represents the same phenomenon.~\cref{t:main} says that a dense $1$-avoiding set has overrepresented distance \sd\, pairs. The proof gives an explicit numerical value of $\gamma$. We prove the theorem under the potentially weaker hypothesis $\delta(A) \geq \delta(A_{\text{Croft}}) - \gamma$.

We then use the graph container method to show that typical unit distance avoiding sets (which have density $\sim\frac12 m_1(\RR^2)$) also have overrepresented distance $\sd$ pairs. 
To make typicality precise, we introduce a little more notation.
We say a subset of $\R^2$ is \textit{locally constant at scale $\frac{1}{N}$} if it is a union of $\frac{1}{N}\times\frac{1}{N}$ grid squares $[j/N, (j+1)/N] \times [k/N, (k+1)/N]$. We say it is $K$-periodic if it is periodic with respect to $K\Z^2$. There are finitely many sets that are both $\frac{1}{N}$-constant and $K$-periodic. 

\begin{thm}\label{t:main-container} 
Let $\gamma$ be the constant from \cref{t:main}. The following holds for all $K > K_0$ and all $N > N_0(K)$. Among all $\frac{1}{N}$-constant and $K$-periodic sets that are $1$-avoiding, pick a set $A$ uniformly at random. With high probability, we have that
\begin{align*}
    s(\sd; A) \geq 1+\gamma/2. 
\end{align*}
\end{thm}

The main ingredient in the container method is a supersaturation principle for the unit distance graph on $\R^2$, originally due to Bourgain and further developed by \cite{BUKH08}. Roughly speaking, if a subset $A$ of $\RR^2$ has density $\delta(A) \ge m_1(\RR^2) + \eps$, then $A$ has at least a $\gamma(\eps) > 0$ portion of distance $1$ pairs. The supersaturation principle comes from \textit{compactness} of the functional 
\begin{equation*}
    A \mapsto f^{\circ}(1; A),
\end{equation*}
which in turn uses the curvature of the $\ell^2$ norm. 

\subsection*{Outline}

In \S\ref{s:prelim} we review some preliminary results that will be used in the rest of the paper, including the supersaturation and compactness properties of unit distance avoiding sets. In \S\ref{s:containers} we prove Theorem \ref{t:main-container} (overrepresentation of distance \sd\, pairs for typical $1$-avoiding sets) using Theorem \ref{t:main} (overrepresentation for dense $1$-avoiding sets). The main idea of the container method is that sampling a uniformly random $1$-avoiding set looks like first sampling a $1$-avoiding set with near maximum density, and then sampling a uniformly random subset. In \S\ref{s:lp} we prove Theorem \ref{t:main} using a linear program. Finally, in \S\ref{s:more_obs} we discuss more structural observations and questions about $1$-avoiding sets. 

\subsection*{Acknowledgements}
Alex Cohen and Nitya Mani were supported by a Hertz Graduate Fellowship and the NSF Graduate Research Fellowship Program. Thanks to Henry Cohn for interesting conversations. 

\section{Preliminaries}\label{s:prelim}

\subsection{Notation}
We use the inequality notation $A\lesssim B$ to mean there is a universal constant $C$ so that $A \leq CB$.

Fix a large integer $K > 0$ and let $\TT_K^2 = \R^2 / (K \Z^2)$ be a $K\times K$ torus. We denote by $\widehat {\TT_K^2} = \frac{2\pi}{K}\Z^2$ the Fourier dual. We work with the normalized probability measure on $\TT_K^2$ and the counting measure on $\widehat{\TT_K^2}$. For positive integer $L$ we let $[L] := \{1, 2, \ldots, L\}.$

For $f, g \in L^2(\TT_K^2)$ we denote the inner product by
\begin{align*}
    \langle f, g\rangle = \frac{1}{K^2} \int_{\TT_K^2} f(x)\overline g(x)\, dx,
\end{align*}
and we use the Fourier transform
\begin{align*}
    \widehat f(\xi) = \langle f, e^{i \xi \cdot x}\rangle\qquad \text{for $\xi \in \widehat{\TT_K^2}$.}
\end{align*}
We have a Fourier inversion formula
\begin{align*}
    f(x) = \sum_{\xi \in \widehat{\TT_K^2}} \widehat f(\xi) e^{i\xi \cdot x}
\end{align*}
and Plancherel formula
\begin{align*}
    \langle f, g\rangle = \sum_{\xi \in \widehat{\TT_K^2}} f(\xi) \widehat g(\xi).
\end{align*}
Let $A \subset \TT_K^2$ be a set with density $\delta(A) = \frac{1}{K^2} \Area(A)$. The autocorrelation function of $A$ is 
\begin{equation}
    f(x) = \delta(A\cap (A-x)) = \langle \ind_A, \ind_{A-x}\rangle,\quad x \in \TT_K^2,
\end{equation}
and the pair correlation function is 
\begin{equation}
    f^{\circ}(r) = \fint_0^{2\pi} f(rv_{\theta})\, d\theta,\quad v_{\theta} = (\cos \theta, \sin \theta),
\end{equation}
here $\fint_0^{2\pi} = \frac{1}{2\pi} \int_0^{2\pi}$. We interpret $f^{\circ}(r)$ as the density of pairs at distance $r$, and for a randomly chosen set we expect $f^{\circ}(r) \approx \delta(A)^2$. We define $$s(r; A) := \frac{f^{\circ}(r)}{\delta(A)^2},$$ so that $s(r)\approx 1$ for a typical set. Instead of considering sets that literally avoid unit distances we will consider sets with $s(1; A) = 0$, meaning the set of unit distances in $A$ have measure $0$. By a slight abuse of notation, we will also refer to such sets as $1$-avoiding.

Let 
\begin{equation}\label{eq:besselJ0}
    J_0(r) = \fint_0^{2\pi} e^{ir\cos \theta}\, d\theta
\end{equation}
be the Bessel function of the first kind with parameter $0$. If $\mu_{S^1}$ is the normalized measure on a unit circle, then $\widehat \mu_{S^1}(\xi) = J_0(|\xi|)$. We have the Fourier expansion
\begin{equation}\label{eq:fourier_synth_fcirc}
    f^{\circ}(r) = \sum_{\xi \in \widehat{\TT_K^2}} J_0(r|\xi|) |\widehat \ind_A(\xi)|^2.
\end{equation}
Split up $\TT_K^2$ into a grid of $\frac{1}{N}\times \frac{1}{N}$ squares $\left\{Q_j\right\}_{j=1}^{(NK)^2}$. Say $A$ is locally constant at scale $\frac{1}{N}$ if it is a union of closed grid squares,
\begin{equation*}
    A = \bigcup_{j \in [S]} Q_j, \quad S \subset [(NK)^2]
\end{equation*}
In the rest of the paper we take $K$ to be a fixed large constant and let $N$ go to infinity. We denote 
\begin{align*}
    {n\choose \leq m} = {n\choose 0} + {n\choose 1} + \dots + {n\choose m}. 
\end{align*}

\subsection{Approximation}

\begin{defn}
Let $m_{1}(\TT_K^2)$ be the maximum density of a unit distance avoiding set in $\TT_K^2$.
Let $m_{1,N}(\TT_K^2)$ be the maximum density of a unit distance avoiding set in $\TT_K^2$ that is locally constant at scale $\frac{1}{N}$.
\end{defn}

\begin{lemma}\label{lem:approximation}
We have $m_{1,N}(\TT_K^2) \to m_{1}(\TT_K^2)$ as $N \to \infty$. 
\end{lemma}
\begin{proof}
Let $A \subset \TT_K^2$ and suppose $s(1; A) = 0$. For any $\varepsilon > 0$ we will find some $N > 0$ and a set $B \subset \TT_K^2$ which is locally constant at scale $N$, avoids unit distances, and has $\delta(B) \geq \delta(A) - \varepsilon$.

Let $\Pi_N: L^2(\TT_K^2) \to L^2(\TT_K^2)$ be the projection onto the linear subspace of functions that are constant on $\frac{1}{N}\times \frac{1}{N}$ boxes. Explicitly,
\begin{align*}
    \Pi_N f(x) = \fint_{Q_x} f(y)\, dy
\end{align*}
where $Q_x$ is the $\frac{1}{N}\times \frac{1}{N}$ grid square containing $x$ and $\fint_{Q_x}$ is the average value. For any $f \in L^2(\TT_K^2)$ we have $\Pi_N f \to f$ in $L^2$ as $N \to \infty$. 
Let $N > 0$ be large enough that $\| \Pi_N \ind_A - \ind_A\|_2 < \varepsilon$, where $\varepsilon > 0$ is a parameter to be chosen later. Let
\begin{align*}
    \delta_Q(A) &= \fint_Q \ind_A\, dx. 
\end{align*}
We have
\begin{align*}
    \| \Pi_N \ind_A - \ind_A\|_2^2 &= \sum_Q  \frac{1}{N^2}(\delta_Q(A) - \delta_Q(A)^2)
\end{align*}
For any $\alpha > 0$,
\begin{equation}\label{eq:num_Q_del_dif}
    \#\left\{Q\, :\, \delta_Q(A) - \delta_Q(A)^2 > \alpha\right\} < \frac{N^2 \varepsilon^2}{\alpha}. 
\end{equation}
Let $s > 0$ be a parameter to be chosen later. 
Let
\begin{align*}
    \mc Q_s = \Bigl\{Q\text{ a $\frac{1}{N}$-grid square}\, :\, \delta_Q(A) > 1-s\Bigr\}.
\end{align*}
By \eqref{eq:num_Q_del_dif}, 
\begin{align*}
    \delta(A) \leq \frac{|\mc Q_s|}{(NK)^2} + \frac{\varepsilon^2}{\alpha K^2} + \frac{\alpha}{s}.
\end{align*}
Define
\begin{align*}
    B = \bigcup_{Q \in \mc Q_s} (1-\beta) \cdot Q,
\end{align*}
where $\beta$ is a small number to be chosen later and $(1-\beta) \cdot Q$ is a central dilation.
We show $B$ is unit distance avoiding. Indeed, let $Q_1$, $Q_2$ be two grid squares such that $\delta_{Q_1}(A) > 1-s$ and $\delta_{Q_2}(A) > 1-s$, and $(1-\beta) \cdot Q_1$, $(1-\beta) \cdot Q_2$ have a unit distance between them. As long as $s$ is small enough relative to $\beta$, we find that the density of unit distances in $A$ between $Q_1$ and $Q_2$ is positive. 

For any $c > 0$, we can choose $\varepsilon, s, \alpha, \beta$ small enough that $\delta(B) > \delta(A) - c$. If we choose $\beta = 1/M$ for some integer $M$, the set $B$ is locally constant at scale $\frac{1}{2NM}$. 
\end{proof}

\begin{lemma}\label{lem:periodic_approx}
We have $m_1(\TT_K^2) \to m_1$ as $K\to \infty$.
\end{lemma}
\begin{proof}
Let $A \subset \R^2$ be a unit distance avoiding set. For any $\varepsilon > 0$, for large enough $K$ we can find a $K\times K$ square $S$ so that 
\begin{equation*}
    \delta(A\cap S) = \frac{1}{K^2} \cdot \text{Area}(A\cap S) \geq \overline{\delta(A)} - \varepsilon.
\end{equation*} 
Now set $\widetilde A \subset \TT_K^2$ to be $A\cap \widetilde S$, where $\widetilde S$ is a smaller square of side length $K-1$ centered inside $S$. We have $\widetilde A \subset \TT_K^2$ is unit distance avoiding and
\begin{align*}
    \delta(\widetilde A) \geq \overline{\delta(A)} - \frac{4}{K}.
\end{align*} 
Thus, by tiling $\RR^2$ by copies of $\widetilde{A}$, we see that we can find periodic unit distance avoiding sets with density tending to $\overline{\delta(A)}$. 
\end{proof}

\subsection{Compactness and supersaturation}
Our intuition is that dense unit distance avoiding sets are a union of large scale blocks. Unfortunately, we don't know how to prove such a strong structural statement. Bourgain~\cite{BOU86}, and later Bukh~\cite{BUKH08}, proved a compactness property for unit distance avoiding sets which gives some weaker structural information. The compactness principle shows that maximum unit distance avoiding sets have some uniform regularity---for instance, they cannot look like white noise on small scales. The important consequence for us is supersaturation in the unit distance graph. We include a self-contained proof of compactness and supersaturation for the reader's convenience. 

Let $\Phi \in L^2(\TT_K^2)$ be a function, $f$ its autocorrelation function, and $f^{\circ}$ the radialized autocorrelation function. The case we care about is $\Phi = \ind_A$ for some set $A$.
Recall the Fourier expansion \eqref{eq:fourier_synth_fcirc}
\begin{align*}
    f^{\circ}(r) &= \sum_{\xi \in \widehat{\TT_K^2}} J_0(r|\xi|)\, |\widehat \Phi(\xi)|^2.
\end{align*}
By the stationary phase estimate, $J_0(|\xi|) \lesssim |\xi|^{-1/2}$.
\begin{lemma}\label{l:conv}
Suppose $\Phi_j \in L^2(\TT_K^2)$ is a uniformly bounded sequence which weakly converges in $L^2$ to $\Phi \in L^2(\TT_K^2)$. Let $f_j^{\circ}$ be the radialized autocorrelation function of $\Phi_j$ and $f^{\circ}$ the radialized autocorrelation function of $\Phi$. We have $f_j^{\circ}(r) \to f^{\circ}(r)$ for every $r$. 
\end{lemma}
\begin{proof}
We have
\begin{align*}
   |f_j^{\circ}(r) - f^{\circ}(r)| &\lesssim \sum_{\xi \in \widehat{\TT_K^2}, |\xi| \leq M} \Bigl|\, |\widehat \Phi(\xi)|^2 - |\widehat \Phi_j(\xi)|^2\, \Bigr| + (\| \Phi_j \|_2^2 + \| \Phi \|_2^2) \sup_{\xi \in \widehat{\TT_K^2}, |\xi| > M} |J_0(r|\xi|)|. 
\end{align*}
Because of the weak convergence $\Phi_j \rightharpoonup \Phi$, we have $\widehat \Phi_j(\xi) \to \widehat \Phi(\xi)$ for any $\xi \in \widehat{\TT_K^2}$, so the first term above goes to zero as $j\to \infty$ for any fixed $M$. Because $J_0(\xi) \to 0$ as $\xi \to \infty$ and $\| \Phi_j \|_2$ is uniformly bounded, the second term goes to zero as $M\to \infty$. 
\end{proof}

\begin{rem}\label{r:curvature}
The Bessel function $J_0(\xi) = \widehat {\mu_{| x |_2 = 1}}(\xi)$ decays as $\xi\to \infty$ because of the curvature of the unit circle. If we replaced unit distances in the $\ell_2$ norm with the $\ell_{\infty}$ norm then the unit circle would be replaced with the unit square. The unit square is not curved, and $\widehat {\mu_{| x |_{\infty} = 1}}(\xi)$ does not decay as $\xi \to \infty$, so the above proof breaks down. 

As it turns out, the conclusion is not true for the $\ell_{\infty}$ norm. Let
\begin{equation*}
    A_N = \{(x, y) \in \TT_K^2\, :\, \lfloor Nx \rfloor , \lfloor N y \rfloor \text{ are even}.\}
\end{equation*}
Let $g_N = \ind_{A_N}$. Then
\begin{equation*}
    \text{Density of distance 1 pairs in $\ell_{\infty}$} = \fint_{| x |_{\infty} = 1} f^{\circ}(x)\, d\mu(x) = \begin{cases}
        \frac{1}{2} & \text{$N$ is even,} \\
        0 & \text{$N$ is odd}.
    \end{cases}
\end{equation*}
On the other hand, $g_N \rightharpoonup \frac{1}{4}$. Even though $g_N$ has a weak limit, the densities of distance one pairs do not converge.
\end{rem}

The following supersaturation principle follows quickly from compactness.
\begin{lemma}\label{l:supersat}
Let $K > 0$ be fixed. For all $\eps > 0$, there exists a $\gamma = \gamma(\varepsilon, K) > 0$ such that if $\delta(A) \ge m_{1}(\TT_K^2) + \eps$ then $s(1; A) > \gamma$.
\end{lemma}
\begin{proof}
Suppose the conclusion does not hold. Let $A_j$ be a sequence of sets with $s(1, A_j) \to 0$ and $\delta(A_j) > m_1(\TT_K^2) + \varepsilon$. Extract a weakly convergent subsequence $\ind_{A_{j_k}} \rightharpoonup g$. Then $0 \leq g \leq 1$ and $s(1, g) = \lim_{k\to \infty} s(1, A_{j_k}) = 0$. Let $A = \ind_{g > 0}$. Then 
\begin{align*}
    s(1, A) = \sup_{c > 0} s(1, \ind_{g > c}) \leq \sup_{c > 0} c^{-2}s(1, g) = 0,
\end{align*}
so $A$ is a unit distance avoiding set. But $\int g \geq m_1(\TT_K^2) + \varepsilon$ and $\int g \leq \delta(A)$, so $\delta(A) \geq m_1(\TT_K^2)+\varepsilon$. This is a contradiction. 
\end{proof}

\section{Container lemma}\label{s:containers}
The method of graph containers, initially developed and applied by Kleitman and Winston in 1980~\cite{KW80} and Sapozhenko~\cite{SAP03}, gives an algorithmic approach to identify and enumerate the independent sets of a ``well behaved'' graph. A consequence of such a \textit{container lemma} is that the independent sets of such a graph lie in one of a relatively small number of \textit{containers}, sparse sets that are not too much larger than the size of a maximum independent set. This method was later extended to studying hypergraph independent sets independently by Saxton and Thomason~\cite{ST15} and Balogh, Morris, and Samotij~\cite{BMS18} in 2015. In addition to enabling enumeration of independent sets, such lemmas often give strong characterizations of the structure of the independent sets of a given graph or hypergraph.

We apply a graph container lemma in conjunction with a linear-programming based stability result to arrive at a graph version of~\cref{t:main-container}. To state this version, we will need to define a graph that encodes unit-distance avoiding sets.

\begin{notn}
We define a graph $G = (V, E) = \mcG(N, K)$, with vertex set $V = \{Q_j\}_{j \in [(NK)^2]}$, which represents the $(NK)^2$ many $1/N \times 1/N$ squares that arise by placing a scale $1/N$ square grid onto $\TT_K^2.$ We let 
$$E = \left\{ (Q_i, Q_j) \mid \text{ there exists } x \in Q_i, y \in Q_j,\, |x - y| = 1 \right\},$$
i.e. there is a pair of points, one from $Q_i$ and one from $Q_j$, that are at unit distance. By a slight abuse of notation, we overload $Q_i$ to both refer to the vertex in graph $G$ as well as the associated $1/N \times 1/N$ square $Q_i \subset \TT_K^2.$ We think of $K$ as a fixed large constant and $N$ going to infinity.

In this setup, the independent sets of $G$, denoted $\mc I(G)$, are in correspondence with the unit distance avoiding sets in $\TT_K^2$ that are locally constant at scale $\frac{1}{N}$. 

For any graph $G = (V, E)$ we let $e(G) = |E|$ and $v(G) = |V|$. We let $\Delta(G)$ denote the maximum degree of $G$ and let $\overline{d}(G)$ denote the average degree of $G$. If $A \subset V$, we let $G[A]$ denote the induced subgraph of $G$ on vertex subset $A.$
\end{notn}

We begin by making a pair of basic observations about the graph $G = \mcG(N, K)$.
\begin{obs}\label{o:max-deg}
There exists a constant $C > 0$ such that for $G = \mcG(N, K)$ as above, $\Delta(G) \le C N$.
\end{obs}

\begin{obs}\label{o:subgraph}
Given any $F \subset V(G)$ for $G = \mcG(N, K)$ as defined above, we have the following two properties of the corresponding subset of $\TT_K^2$.
\begin{itemize}
    \item $\delta(F) = \frac{|F|}{N^2 K^2}$
    \item There exists some absolute constant $c_1 > 0$ such that $s(1; F) \le c_1 \frac{e(G[F])}{N^3 K^2}$
\end{itemize}
\end{obs}

We will show the following graphical statement that is equivalent to~\cref{t:main-container}.
\begin{thm}\label{t:discrete}
Let $\gamma$ be the constant from \cref{t:main}. Let $K > 0$, $N > N_0(K)$, and $G = \mcG(N, K)$. The probability that a uniformly random $A \in \mcI(G)$ has $s(\sd; A) > 1 + \gamma/2$ is $\geq 1 - o_N(1)$. 
\end{thm}
The proof of~\cref{t:discrete} employs three key ingredients:
\begin{enumerate}
 \item A container lemma for $\mcI(G)$ (stated below as~\cref{l:cont-main});
 \item The \textit{supersaturation lemma} (\cref{l:supersat}) proved in the previous section;
 \item \cref{t:main}, overrepresentation of distance \sd\, pairs for dense sets. 
\end{enumerate}
 
Here is our graph container lemma, the proof of which follows from the Kleitman-Winston algorithm (see~\cite{KW80,SAM15}.
\begin{lemma}[Weak containers (Theorem 1.6.1~\cite{AS4})] \label{l:weak-cont}
For any $D > 0$ and any graph $G = (V, E)$ with $|V| = n$, there exists $\mcF \subset 2^{V(G)}$ such that the following properties hold:
\begin{enumerate}
    \item $|\mcF| \le {n \choose \le n/D}$
    \item For all $I \in \mcI(G)$, $I \subset F$ for some $F \in \mcF$
    \item For all $F \in \mcF$, $|F| \le \frac{n}{D} + n(D; G)$, where $n(D; G)$ is the maximum number of vertices in an induced subgraph of $G$ with at most $Dn / 4$ edges
\end{enumerate}
\end{lemma}

\begin{cor}\label{c:weak-cont-spec}
For any graph $G = (V, E)$ with $|V| = n$, average degree $\ol{d} = \ol{d}(G)$, and maximum degree $\Delta(G) \le C \overline{d}$, there exists $\mcF \subset 2^{V(G)}$ such that the following properties hold:
\begin{enumerate}
    \item $|\mcF| \le {n \choose \le n/\ol d}$
    \item For all $I \in \mcI(G)$, $I \subset F$ for some $F \in \mcF$
    \item For all $F \in \mcF$, $|F| \le n \left(1 + 1/\ol d - 1/(4C) \right)$ 
\end{enumerate}
\end{cor}
\begin{proof}
We apply \cref{l:weak-cont} with $D = \ol d$.
Consider $U \subset V(G)$ where $e(G[U]) \le \ol d n / 4$. Note that $$e(G[U]) \ge e(G) - \sum_{v \in V \backslash U} d(v) \ge e(G) - \Delta(G) (n - |U|).$$
Since $e(G) = \ol d n / 2$, we can combine the inequalities to find that 
$$\frac12 \ol d n - C \ol d (n - |U|) \le \frac12 \ol d n - \Delta(G) (n - |U|) \le e(G[U]) \le \frac14 \ol d n.$$
Rearranging gives that 
$$|U| \le \left(1 - \frac{1}{4C} \right) n.$$
The result then follows by direct application of~\cref{l:weak-cont} with $D = \ol d$.
\end{proof}

We iterate~\cref{c:weak-cont-spec} in conjunction with~\cref{l:supersat} to obtain our stronger container lemma.
\begin{lemma}[Containers]\label{l:cont-main}
Let $G = \mcG(N, K)$. Fix any $\eps > 0$ and let $C = C (\eps) > 0$ be sufficiently large in $\eps$. Then, for all sufficiently large $N$, 
there is a collection of vertex subsets $\mcF \subset 2^{V(G)}$ with the following properties:
\begin{enumerate}
    \item $|\mcF| \le  2^{ C NK^2 \log(N K)}$; \label{item:bdd_containers}
    \item For all $I \in \mcI(G),$ there exists some $F \in \mcF$ such that $I \subset F$;
    \item For all $F \in \mcF$, either $\delta(F) < m_1(\RR^2) - \eps$ or $s(1, F) \le \eps$\label{item:small_dens}
\end{enumerate}
\end{lemma}
\begin{rmk}
The above result would continue to be true (with a quantitatively different $C$ in~\cref{item:bdd_containers}), if we replaced $\delta(F) < m_1(\RR^2) - \varepsilon$ with the stronger condition $\delta(F) < 1/1000$ in~\cref{item:small_dens} above. When we apply~\cref{l:cont-main}, we will handle ``small'' containers (i.e. those with density $< m_1(\RR^2) - \varepsilon$) separately, so our choice of parameters here does not make much difference.
\end{rmk}
\begin{proof}
Throughout, we assume $N$ is sufficiently large. Fix an arbitrary total order $v_1, \ldots, v_{v(G)}$ on the $N^2K^2$ vertices of $G$. We construct $\mcF$ via the following iterative algorithm.
Initialize $\mcF_0 = \{V\}$. We then sequentially apply the below process for $i \in \NN$.
For each $F_{j} \in \mcF_i$, if $\delta(F_j) > m_1(\RR^2) - \eps$ and $s(1; F) > \eps$, we apply~\cref{c:weak-cont-spec} to $G[F_{j}]$ to arrive at a set of containers $\mcF_{i, j}$. If $\delta(F_j) \le m_1(\RR^2) - \eps$ or $s(1; F) \le \eps$, we let $\mcF_{i, j} = \{F_j\}$. We then set $$\mcF_{i+1} = \bigcup_{j \in \mcF_i} \mcF_{i, j}$$ 
If $\mcF_{i+1} = \mcF_{i},$ we abort and return $\mcF = \mcF_{i}.$ We will show that this process repeats a bounded number of times and yields a family of subsets $\mcF$ that satisfy the conditions of the statement.

Consider some $F_j \in \mcF_i$ where $\delta(F_j) > m_1(\RR^2) - \eps$ and $s(1; F) \ge \eps$. Applying~\cref{o:subgraph}, we find that 
$$\eps \le s(1; F_j) \le c_1 \frac{e(G[F])}{N^3 K^2}$$
Rearranging gives that for constant $c_2 = \frac{2\eps}{c_1} > 0$, we have 
$$\ol d(G[F_j]) = \frac{2 e(G[F_j])}{|F_j|} \ge \frac{2\eps}{c_1} \cdot \frac{N^3 K^2}{N^2 K^2} \ge c_2 N$$
We further have (see~\cref{o:max-deg}) $\Delta(G[F_j]) \le \Delta(G) \le C N$ for some constant $C > 0$, and thus $\Delta(G[F_j]) \le c_3 \cdot \ol d(G[F_j])$ (for $c_3 = C/c_2$). Therefore, applying~\cref{c:weak-cont-spec} gives a family $\mcF_{i, j}$ of containers for the independent sets of $G[F_j]$ such that $|\mcF_{i, j}| \le {|F_j| \choose \le |F_j| /(c_2 N)}$ and where for each $F \in \mcF_{i, j}$, $$|F| \le |F_j| \left(1 + \frac{1}{c_2 N} - \frac{1}{4 c_3} \right) \le |F_j| \left(1 - \underbrace{\frac{1}{8 c_3}}_{=: c_4} \right),$$
and thus $\delta(F) \le (1 - c_4) \delta(F_j)$. Let $i_0$ be a sufficiently large constant such that $(1 - c_4)^{i_0} < m_1(\RR^2) - \eps$. Then, for some $i \le i_0$, we must have $\mcF_{i + 1} = \mcF_{i}$ i.e. the process terminates after at most $i_0$ iterations. Further, for consecutive iterations, we have that for some constant $c_5 > 0$
$$
\frac{\mcF_{i+1}}{\mcF_i} \le {|F_j| \choose \le |F_j|/(c_2 N)} \le {(1 - c_4)^{i_0} N^2 K^2 \choose \frac{(1 - c_4)^{i_0}}{c_2} N K^2} \le 2^{c_5 NK^2 \log (N K)},
$$
and since $i_0$ is a constant, we have $|\mcF| \le 2^{c_6 NK^2 \log (N K)}$ for some constant $c_6 > 0$, thereby verifying Condition 1. Condition 2 follows by the guarantee of~\cref{c:weak-cont-spec}, and Condition 3 follows by our stopping condition for the recursive application of~\cref{c:weak-cont-spec} in constructing $\mcF.$
\end{proof}

We are now ready to prove~\cref{t:discrete}.

\begin{proof}[Proof of~\cref{t:discrete}]
Let $G = \mcG(N, K)$. Take $\gamma > 0$ as in~\cref{t:main}. 
Let $K > K_0$ be a large constant and $N > N_0(K)$. 
As long as $K$ and $N$ are large enough, Lemma \ref{lem:approximation} and Lemma \ref{lem:periodic_approx} together imply $\alpha(G) \ge (m_1(\RR^2) - o_N(1)) v(G)$ and thus $|\mcI(G)| \ge 2^{(m_1(\RR^2) - o_N(1)) v(G)}$. We will upper bound the number of independent sets $A \in \mcI(G)$ with $s(\sd; A) < 1 + \gamma/2$.

Pick a small constant $\varepsilon_1 > 0$ to be chosen later, and let $\mcF \subset 2^{V(G)}$ be the set of containers obtained by applying~\cref{l:cont-main}. Note that for all $I \in \mcI(G), I \subset F$ for some $F \in \mcF$. We partition $\mcF = \mcF_1 \cup \mcF_2$, where 
$$\mcF_1 := \left\{F \in \mcF : \delta(F) \le m_1(\RR^2) - \eps_1 \right\}, \quad \mcF_2 = \mcF \backslash \mcF_1.$$
For all $F \in \mc F_2$, we have (choosing $\varepsilon_1 < \gamma)$
\begin{itemize}
    \item $\delta(F) \geq m_1(\RR^2) - \varepsilon_1 \geq m_1(\RR^2) - \gamma $
    \item $s(1; F) \leq \varepsilon_1 \leq \gamma$
    \item By~\cref{t:main} combined with the lower bound on $\delta(F)$ and upper bound on $s(1; F)$, we find that $s(\sd; F) \ge 1 + \gamma$.
    \item Because $s(1; F) \leq \eps_1$, supersaturation (\cref{l:supersat}) implies $\delta(F) \leq m_1(\TT_K^2) + \widetilde \eps_2 \leq m_1(\RR^2) + \eps_2$ where $\eps_2 = \eps_2(\eps_1, K_0)$ goes to zero as $\eps_1 \to 0$ and $K_0 \to \infty$. 
\end{itemize}
With this partition, we have the following:
\begin{align*}
\left| \left\{ A \in \mcI(G) : s(\sd; A) < 1 + \gamma/2 \right\} \right| \le \sum_{F \in \mcF_1} 2^{|F|} + \sum_{F \in \mcF_2} \left| \left\{ A \subset F : s(\sd; A) < 1 + \gamma/2 \text{ and } A \in \mcI(G) \right\} \right|
\end{align*}
We can crudely bound the first term by noting that $|\mcF_1| \le |\mcF|$ and that for $N$ sufficiently large, we have the following:
\begin{align*}
\sum_{F \in \mcF_1} 2^{|F|} \le |\mcF| 2^{(m_1(\RR^2) - \eps_1)v(G)} \le 2^{C NK^2 \log (NK) + (m_1(\RR^2) - \eps_1) N^2 K^2} \ll  |\mcI(G)|,
\end{align*}
For the second term, note that for $A \subset V(G)$, the function $A \mapsto s(\sd; A)$ is $\frac{C_1}{v(G)}$-Lipschitz for some absolute constant $C_1 > 0$. 
Fix some $F \in \mcF_2$.
For a uniformly random subset $A \subset F$, $\EE[s(\sd; A)] = s(\sd; F) \ge 1 + \gamma$. By concentration of measure (for instance one can use the Azuma-Hoeffding inequality),
$$\P(s(\sd; A) \le 1 + \gamma/2) \le \exp\left(-\frac{1}{C_1} \gamma^2 v(G) \right)$$
Therefore, 
\begin{align*}
\sum_{F \in \mcF_2} \left| \left\{ A \subset F : s(\sd; A) < 1 + \frac{\gamma}{2} \text{ and } A \in \mcI(G) \right\} \right| &\le |\mcF_2| 2^{(m_1(\RR^2) + \eps_2) v(G)} \exp\left(-\frac{1}{C_1} \gamma^2 v(G) \right) \\
&=  2^{CNK^2 \log(NK) + (m_1(\RR^2) + \eps_2 - \frac{1}{C_1} \gamma^2) v(G)}. 
\end{align*}
Choose $K_0$ large enough and $\eps_1$ small enough that $\eps_2(\varepsilon_1, K_0) - \frac{1}{C_1} \gamma^2 < -\eps_3$ for some fixed $\eps_3 > 0$. Because $|\mcI(G)| \ge 2^{(m_1(\RR^2) - o_N(1)) v(G)}$, we find 
\begin{align*}
    \Pr[s(\sd; A) \leq 1+\gamma/2] \geq 1-o_N(1)
\end{align*}
as desired. 

\end{proof}

\section{Excess pair correlation via linear programming}\label{s:lp}
In this section we apply the linear programming approach pioneered by~\cite{SZE84,FV10} 
to verify~\cref{t:main}. These LP bounds come from studying the Fourier expansion of a radialized version of the \textit{autocorrelation function} $f(x) = \delta(A \cap (A - x)) = \langle \ind_A, \ind_{A - x}\rangle$ for measurable $A \subset \TT_K^2$. We first recall some general properties of $f(x)$, as observed in~\cite{SZE84,FV10,AM22, KMOFR16,ACMVZ22}.
\begin{itemize}
    \item[(D)] $f(0) = \delta(A)$
    \item[(G)] If $G$ is any finite graph with vertex set $V(G) \subset \R^2$, then (c.f.~\cite{FV10})
    \begin{equation*}
        \sum_{x \in V(G)} f(x) - \sum_{(x, y) \in E(G)} f(x - y) \le \alpha(G) f(0).
    \end{equation*}
    In particular, if $G$ is a subgraph of the unit distance graph, then 
    \begin{equation*}
        \sum_{x \in V(G)} f^{\circ}(x)  \le \alpha(G) f(0) + f^{\circ}(1) |E(G)| 
    \end{equation*}

    \item[(CT)] Let $\theta \in [0, 2\pi)$, and let $G_1(\theta), G_2(\theta)$ be the graphs defined in~\cite[page 1250]{AM22}. We have 
    \begin{equation*}
         \sum_{(x, y) \in {V(G_1) \choose 2}} f^{\circ}(|x - y|) - \sum_{x\in V(G_2)} f^{\circ}(|x|) \ge 5\delta(A) - 1  - c_{CT}f^{\circ}(1).
    \end{equation*}
    for some absolute $c_{CT} > 0$. 
\end{itemize}
The most complicated constraint is (CT). A version of this was proved in~\cite[page 1250]{AM22} under the hypothesis $f^{\circ}(1) = 0$. The above version has a quantitative dependence on $f^{\circ}(1)$ and follows from the same proof. 
Since $f(x) = \langle \ind_A, \ind_{A - x} \rangle$, the radialized autocorrelation function $f^{\circ}(r)$ has a Fourier expansion (per~\eqref{eq:fourier_synth_fcirc})
\begin{equation}\label{e:kappa}
f^{\circ}(r) = \sum_{\xi \in \widehat{\TT_K^2}} J_0(r|\xi|) \widehat{f}(\xi) = \sum_{t \in \RR_+} J_0(t r) \underbrace{\sum_{|\xi| = t} | \widehat{\ind_A}(\xi)|^2}_{=: \kappa(t)} = \sum_{t \in \RR_+} J_0(t r) \kappa(t),
\end{equation}
where $J_0(z)$ denotes the zeroth Bessel function of the first kind (see \eqref{eq:besselJ0}). Here $\xi$ ranges over $\frac{2\pi}{K}\Z^2$, so $|\xi|$ ranges over a discrete set of real numbers and $\kappa(t)$ takes values on this discrete set of $t$. The notation $\sum_{t\in \R_+}$ refers to a countable sum over this suppressed index set.
We have the following relations involving $\kappa$.
\begin{itemize}
    \item[(D)] $\kappa(0) = \widehat{f}(0) = |\widehat{\ind_A}(0)|^2 = \delta(A)^2$
    \item[(F1)] $\kappa(t) \ge 0$ for all $t \in \RR$ 
    \item[(F2)] $\sum_{t \in \RR_+} \kappa(t) = \sum_{\xi \in \widehat{\TT_K}^2} \widehat{f}(\xi) = f(0) = \delta(A)$
    \item[(A1)] $f^{\circ}(1) = \sum_{t \in \RR_+} \kappa(t) J_0(t)$
    \item[(A2)] $f^{\circ}(\sd) = \sum_{t \in \RR_+} \kappa(t) J_0(\sd t)$
    \item[(G)] If $G$ is any finite graph, then
    \begin{equation*}
        \sum_{t\in \R_+} \kappa(t)\Bigl (\sum_{x\in V(G)} J_0(t|x|) - \sum_{(x,y) \in E(G)} J_0(t| x-y|)\Bigr) \leq \alpha(G)\delta(A).
    \end{equation*}

    \item[(CT)] For any $\theta \in [0,2\pi)$ we have
    \begin{equation*}
         \sum_{t\in \RR_+}\kappa(t) \Bigl(\sum_{(x, y) \in {V(G_1(\theta)) \choose 2}} J_0(t|x - y|) - \sum_{x\in V(G_2(\theta))} J_0(t|x|)\Bigr) \ge 5\delta(A) - 1  - c_{CT}f^{\circ}(1).
    \end{equation*}
\end{itemize}
With the exception of (D), these are all linear constraints on the $\{\kappa(t)\}, \delta(A)$, which motivated a line of work (c.f.~\cite{SZE84,FV10,AM22, KMOFR16,ACMVZ22}) proving upper bounds on $m_1(\RR^2)$ by solving linear programs.
In order to prove~\cref{t:main}, we study how large $\delta(A)$ can be for a set $A$ where both $s(1; A) \le \gamma$ and $s(\sd; A) \le 1 + \gamma$. The linear program dual produces a \textit{witness function} of the following form. 
Let $v_0, v_1, v_{\sd}$, $w_{M_i}$ for $i \in [3]$, $w_{T_i}$ for $i\in [10]$, and $w_{\theta}$ for $\theta$ in a finite collection of angles $\theta \in [0, 2\pi]$ be a collection of nonnegative coefficients. Define witness function 
\begin{align}\label{eq:witness}
    W(t) &= v_0 + v_1 J_0(t) + v_{\sd} J_0(\sd t) + \sum_{i \in [3]} w_{M_i} \sum_{x\in V(M_i)} J_0(t|x|) + \sum_{i\in [10]} w_{T_i} \Bigl(\sum_{x\in V(T_i)} J_0(t|x|) - \sum_{\{x,y\} \in E(T_i)} J_0(t| x - y |)\Bigr) \notag \\ 
    &\qquad - \sum_{\theta} w_{\theta} \Bigl(\sum_{\{x,y\} \in {V(G_1(\theta)) \choose 2}} J_0(t|x-y|) - \sum_{x\in V(G_2(\theta))} J_0(t|x|)\Bigr) 
\end{align}
The terms in $W(t)$ correspond to constraints on $\kappa(t)$. 

\begin{lemma}
Let $v_0, v_1, v_{\sd}, w_{M_i}, w_{\theta}$ be a choice of non-negative coefficients. 
Suppose that for $W(t)$ as defined above, $W(0) \geq 1$ and $W(t) \geq 0$ for all $t \geq 0$. Let $A \subset \RR^2$ be a periodic set with $\delta(A) = \delta$, $s(1; A) \leq \gamma$ and $s(\sd; A) \leq 1+\gamma$. Then for an absolute constant $C > 0$
\begin{align*}
    -(1-v_{\sd})\delta^2 + \delta \left[v_0 + 2\sum_{i\in [3]} w_{M_i} + \sum_{i\in [10]} w_{T_i} - 5\sum_{\theta} w_{\theta}\right] + \sum_{\theta} w_{\theta} + C\gamma[\text{sum of coefficients}] \geq 0.
\end{align*}
Here `sum of coefficients' is $v_0 + v_1 + v_{\sd} +  \sum_i w_{M_i} + \sum_{\theta} w_{\theta}$. 
\end{lemma}
\begin{proof}
Let $A$ be $K$-periodic. Let
\begin{align*}
    \widetilde \kappa(t) = \frac{\kappa(t)}{\delta}\qquad t \in \left\{|\xi|\, :\, \xi \in \frac{2\pi}{K}\Z^2\right\}. 
\end{align*}
We have 
\begin{align*}
    \sum_{t\in \R} \widetilde \kappa(t) J_0(rt) &= s(r; A) \delta. 
\end{align*}
We have 
\begin{align*}
\delta &= \widetilde{\kappa(0)} \le \sum_{t \ge 0} \widetilde{\kappa}(t) W(t) \\
&\le \sum_{t \ge 0 } \widetilde{\kappa}(t) \Bigg[ v_0 + v_1 J_0(t) + v_{\sd} J_0(\sd t) + \sum_{i \in [3]} w_{M_i} \sum_{x \in V(M_i)} J_0(t |x|) \\
&\qquad + \sum_{i \in [10]} w_{T_i} \left( \sum_{x \in V(T_i)} J_0(t |x|) - \sum_{xy \in E(T_i)} J_0(t|x - y|) \right)  \\
&\qquad - \sum_{\theta} w_{\theta} \left( \sum_{(x, y) \in {V(G_1(\theta)) \choose 2}} J_0(t |x - y|) - \sum_{x\in V(G_2(\theta))}  J_0(t|x|)\right) \Bigg] \\
&=  v_0 \underbrace{\sum_{t \ge 0} \widetilde{\kappa}(t)}_{=1}  +  v_1  \underbrace{\sum_{t \ge 0 } \widetilde{\kappa}(t) J_0(t)}_{\leq \gamma\delta} + v_{\sd} \underbrace{\sum_{t \ge 0 } \widetilde{\kappa}(t) J_0(\sd t)}_{\leq (1+\gamma)\delta} +  \sum_{i \in [3]} w_{M_i}  \underbrace{\sum_{t \ge 0 } \widetilde{\kappa}(t) \sum_{x \in V(M_i)} J_0(t |x|)}_{\le 2+11\gamma\delta(A)} \\
&\qquad + \sum_{i \in [10]} w_{T_i} \underbrace{\sum_{t \ge 0 } \widetilde{\kappa}(t)  \left( \sum_{x \in V(T_i)} J_0(t |x|) - \sum_{xy \in E(T_i)} J_0(t|x - y|) \right)}_{\le 1+3\gamma\delta}  \\
&\qquad - \sum_{\theta} w_{\theta} \underbrace{\sum_{t \ge 0 } \widetilde{\kappa}(t) \left( \sum_{(x, y) \in {V(G_1(\theta)) \choose 2}} J_0(t |x - y|) - \sum_{x\in V(G_2(\theta))}  J_0(t|x|)\right)}_{\ge 5 - 1/\delta - c_{CT} \gamma\delta } \Bigg] \\
&\leq v_0 + v_1\gamma \delta   + v_{\sd} (1+\gamma)\delta + (2+11\gamma\delta) \sum_{i \in [3]} w_{M_i} + (1+3\gamma \delta)\sum_{i \in [10]} w_{T_i} - \sum_{\theta} w_{\theta} (5 - 1/\delta - c_{CT} \gamma \delta)  
\end{align*}
We rearrange to get (using $\gamma \delta \leq \gamma$)
\begin{align*}
    \delta^2 &\leq v_0\delta + v_{\sd} \delta^2 + 2\delta \sum_{i\in [3]} w_{M_i} + \delta \sum_{i\in [10]} w_{T_i} - \sum_{\theta} w_{\theta}(5\delta-1) +\\ 
    &\qquad +\gamma \left[v_1 + v_{\sd} + 11\sum_{i\in [3]} w_{M_i} + 3\sum_{i\in [10]} w_{T_i} + c_{CT} \sum_{\theta} w_{\theta}\right] \\ 
    0 &\leq -(1-v_{\sd})\delta^2 + \delta\left[v_0 + 2\sum_{i\in [3]} w_{M_i} + \sum_{i\in [10]} w_{T_i} - 5\sum_{\theta} w_{\theta}\right] + \sum_{\theta} w_{\theta} + C\gamma[\text{sum of coefficients}].
\end{align*}
\end{proof}

\begin{proof}[Proof of Theorem \ref{t:main}]
We use a linear program to find an explicit choice of coefficients such that $W(0) \geq 1$, $W(t) \geq 0$ for $t\geq 0$, and the maximum root of the quadratic equation 
\begin{align*}
    -(1-v_{\sd})\delta^2 + \delta[v_0 + 2\sum_{i\in [3]} w_{M_i} + \sum_{i\in [10]} w_{T_i} - 5\sum_{\theta} w_{\theta}] + \sum_{\theta} w_{\theta} = 0 
\end{align*}
is some $\delta^* \leq 0.229$. Fix some $\varepsilon > 0$. If $\gamma$ is chosen small enough in terms of $\varepsilon$ and the coefficients, we have that for all $\delta > 0.229+\varepsilon$,
\begin{align*}
    -(1-v_{\sd})\delta^2 + \delta[v_0 + 2\sum_{i\in [3]} w_{M_i} + \sum_{i\in [10]} w_{T_i} - 5\sum_{\theta} w_{\theta}] + \sum_{\theta} w_{\theta} + C\gamma [\text{sum of coefficients}] < 0
\end{align*}
Thus there is some $\gamma > 0$ so that as long as $s(1; A) \leq \gamma$ and $s(\sd; A) \leq 1+\gamma$, we have 
\begin{align*}
    \delta(A) \leq \delta(\text{Croft tortoise}) -\gamma \leq m_1(\R^2) - \gamma.
\end{align*}
\end{proof}

\subsection{Remarks on the numerical solution}
To prove Theorem \ref{t:main} we solve the below feasibility problem, where $\delta^+$ is a fixed upper bound on the maximum density, and thus renders the following optimization problem a linear program\footnote{Example code to solve this feasibility program can be found at \url{https://colab.research.google.com/drive/1Y6gxdoKaDah22Dk4VOa0B1G223peNxul}.}:

\begin{align*}
\text{find}\qquad & v_0, v_{\sd}, w_{M_i}, w_{T_j}, w_{\theta} \\
\text{such that  }    &(\delta^+)^2 (1-v_{\sd}) \geq \delta^+\Bigl(v_0 + 2 \sum_{i\in [3]} w_{M_i} + \sum_{i\in [10]} w_{T_i} - 5\sum_{\theta} w_{\theta}\Bigr) + \sum_{\theta} w_{\theta} \\
&v_0, v_{\sd}, w_{M_i}, w_{T_j}, w_{\theta} \ge 0, \quad \text{for all } i \in [3], j \in [10] \\
& W(0 ) \ge 1 \\
& W(t) \ge 0 \quad \text{for all $t > 0$} 
\end{align*}
Numerically we first find a choice of coefficients where $W(t) > 0$ for all $t$ in a dense grid, and then verify after the fact that $W(t) > 0$ by evaluating $W(t)$ at a much higher resolution than its variations (this is consistent with prior works in this area). More precisely, we observe that $|J_0'(t)| = |J_1(t)| < 0.6$, and that for all coefficients $r$ appearing as $J_0(rt)$ in contributions to $W(t)$, $r < 4$. Noting that we can find a feasible LP solution with the additional constraint that the sum of the coefficients satisfies
$$v_0 + v_1 + v_{\sd} + \sum_i w_{M_i} + 2\sum_j w_{T_i} + 2\sum_{\theta} w_{\theta} \le 15,$$ we see that for any such choice of coefficients
$$|W'(t)| < 4 \cdot \left(v_0 + v_1 + v_{\sd} + \sum_i w_{M_i} + 2\sum_j w_{T_i} + 2\sum_{\theta} w_{\theta} \right) \cdot \max_{t} |J_0'(t)| \le 4\cdot 15 \cdot 0.6 = 36.$$
Then, given some choice of coefficients $v_0, v_{\sd}, w_{M_i}, w_{T_j}, w_{\theta}$ obtained by solving the above feasibility program, we can verify that they define a nonnegative witness function $W(t)$ by checking that $W(t) \ge 0.003$ for all $t = 0.00001 j$, $j\in \Z_{\geq 0}$ and $t \in [0, 20]$, and noting that for $t > 20$ we have $W(t) > 0.1$.

See  ~\cref{f:nearby-lp-autocorr} for a plot of the empirical autocorrelation function corresponding to the primal solution. A plot of the witness function $W(t)$ we constructed along with the associated coefficients can be found below in~\cref{fig:witness}. 

\begin{figure}
    \centering
    \includegraphics[width=0.5\linewidth]{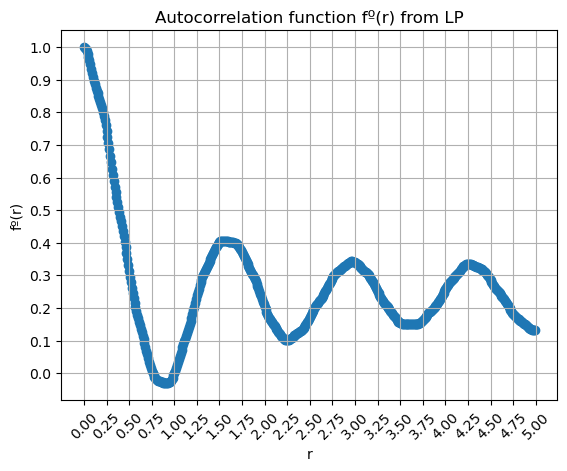}
    \caption{Example autocorrelation function $f^{\circ}(r)$ obtained by solving an approximate primal LP with $\delta^+$ chosen near to the optimal density of the LP (for a true upper bound, the LP is infeasible)}
    \label{f:nearby-lp-autocorr}
\end{figure}

\begin{figure}
    \centering
\includegraphics[width=0.5\linewidth]{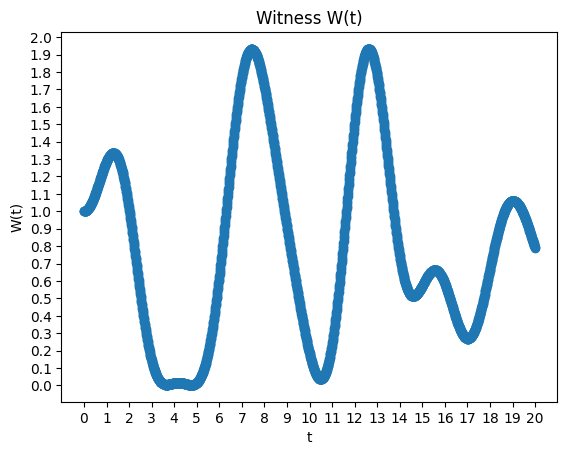}
\begin{tabular}{|c|c|}
\hline 
$v_0$&0.0244 \\
$v_1$ &9.0158\\
$v_{\sd}$ &1.9724 \\
$w_{M_i}$ & $ \left( 0.000949, 0.00394, 0.01952\right)$ \\
$w_{T_i}$&$\big(0.00937, 0.00199, 0.00220, 0.00164, 0.00149, 0.0479,
 0.0925,  0.00203, 0.00231, 0.00316) $\\
$w_{\theta} $ &$\left(0.00140, 0.00202, 0.00438, 0.0898,  0.630\right)$ \\
\hline 
\end{tabular}
    \caption{A plot of the witness function $W(t)$ along with the coefficients (as defined in~\cref{eq:witness}) chosen to certify an upper bound on the density of a $1$-avoiding set with few distance $\sd$ pairs.}
    \label{fig:witness}
\end{figure}

\section{More structural observations}\label{s:more_obs}

In this paper we discuss long range correlations in unit distance avoiding sets. If a set avoids unit distances, it tends to have more than the expected number of distance two pairs. This analysis was motivated by the observation that the known examples of dense unit distance avoiding sets are unions of clusters.
It seems important yet  difficult to prove that maximum unit distance avoiding sets behave in this way, and we'd like to discuss some known results and open questions about this conjectural picture. 

Keleti, Matolcsi, Filho, and Ruzsa~\cite{KMOFR16} defined the following notion. A set $A \subset \R^2$ has \emph{block structure} if 
\begin{equation*}
    A = \bigcup_{j=1}^{\infty} A_j,
\end{equation*}
where $|x-y| < 1$ for $x,y$ in the same block $A_j$, and $|x-y| > 1$ for $x,y$ in different blocks. They gave a nice proof that if $A$ is unit distance avoiding and has block structure, $\delta(A) < 1/4$. First they use the Brunn-Minkowski inequality and the isodiametric inequality to prove that $\delta(A) \leq 1/4$. They then use stability theorems to conclude $A$ has a very special structure if $\delta(A) \geq 1/4 - \varepsilon$, and rule out that case. 

\begin{qn}
Does there exist a set $A$ with block structure and $\delta(A) = m_1$?
\end{qn}

In order to get at this question we ran a computer experiment. We used an optimization package (integer linear programming in Gurobi) to find a large unit distance avoiding set on a $50\times 50$ grid. The output has block structure, even though the optimization package only knows about unit distances in this graph and not nearest neighbor distances. See Figure \ref{fig:opt_sol_50x50}. It would be interesting to run a similar experiment on larger grid sizes. 

\begin{figure}
    \centering
    \includegraphics[width=0.5\linewidth]{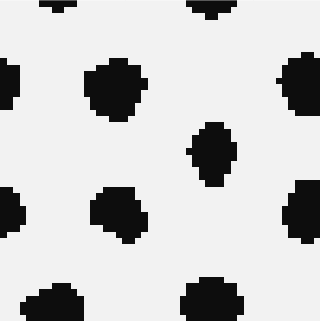}
    \caption{An optimization program finding a large unit-distance avoiding set}
    \label{fig:opt_sol_50x50}
\end{figure}

A positive answer to this question would give very strong structural information about unit distance avoiding sets. A linear programming analysis of the pair correlation function should give at least partial information in this direction. The fact that $A$ comes in clusters is reflected in $f^{\circ}(r)$ being large for small values of $r$, and one could potentially establish this rigorously and quantitatively. 

In another direction, it is interesting to consider a greedy random model for independent sets rather than a uniform random model. Given a graph $G$, one can sample an independent set by adding one vertex at a time, each one uniformly at random, until one obtains a maximal independent set. For discrete unit distance graphs this process tends to produce very sparse unit distance avoiding sets, see Table~\ref{tab:greedy_unit_distance_avoiding_sets} for an example. Nevertheless, they still have some amount of clustering, and some excess large scale correlations---see Figure \ref{fig:pair_corr_greedy} for the pair correlation function of a greedy unit distance avoiding set. It would be interesting to establish these observations rigorously.
\begin{table}[h]
\centering 
\begin{tabular}{ccc} 
\includegraphics[width=0.3\linewidth]{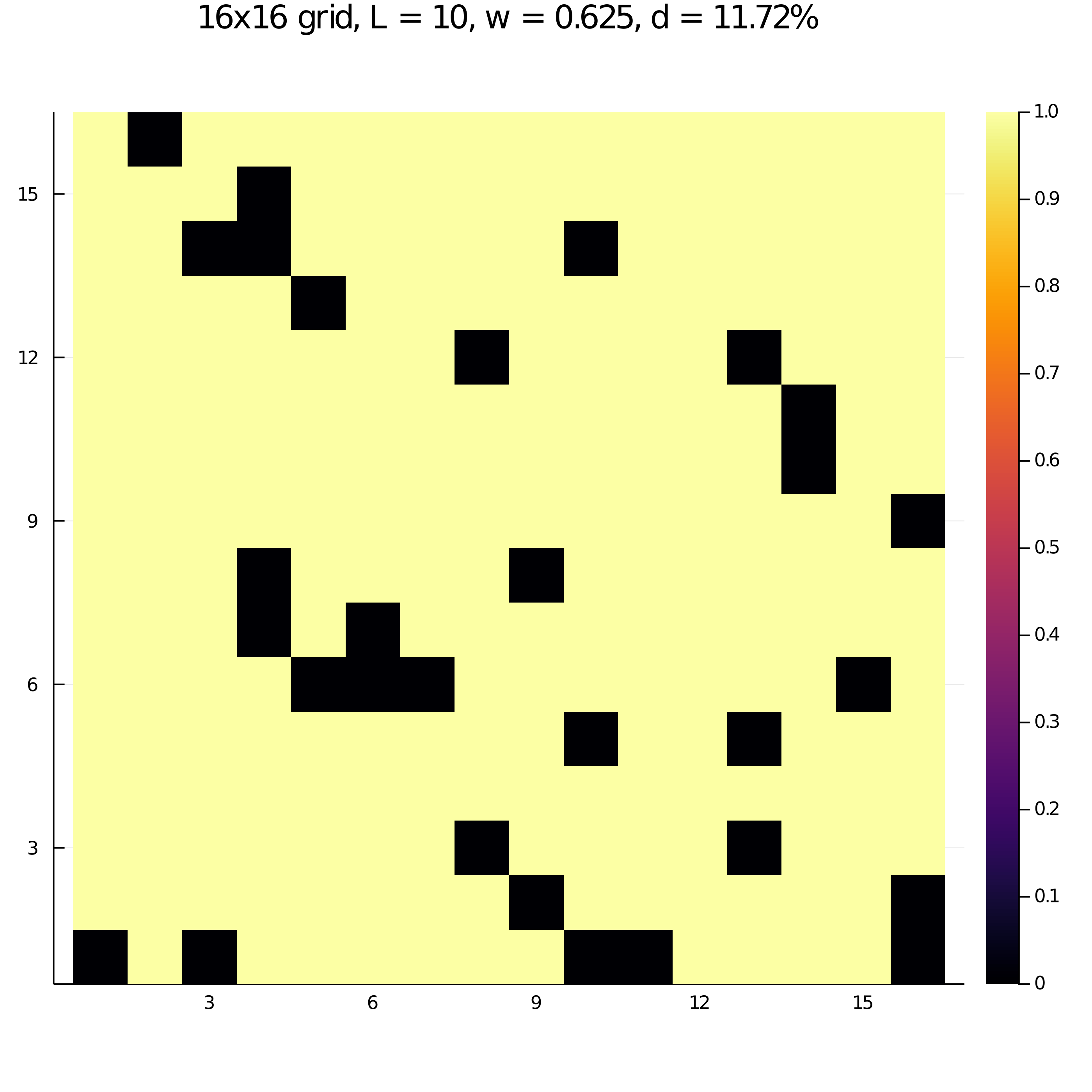} &
\includegraphics[width=0.3\linewidth]{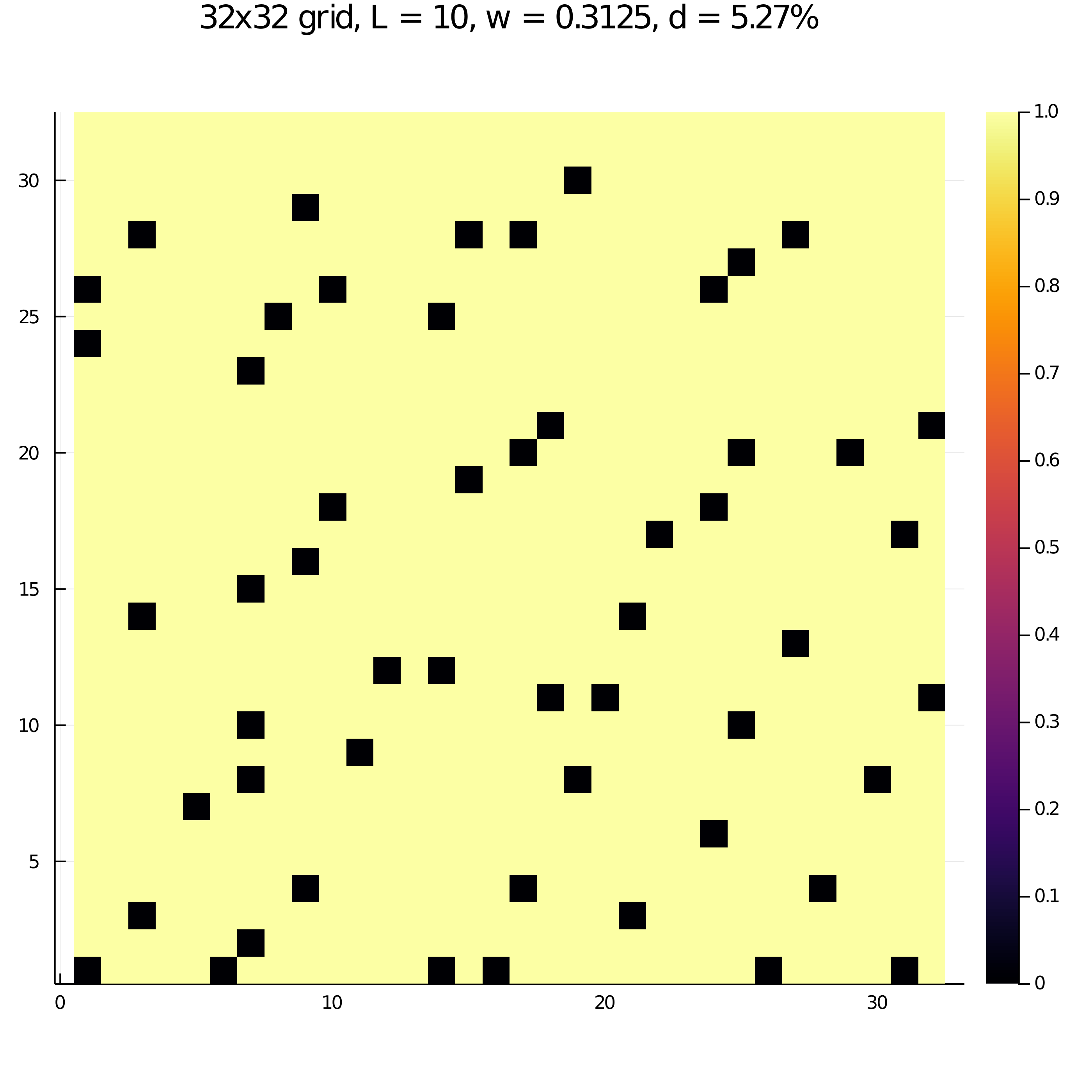} &
\includegraphics[width=0.3\linewidth]{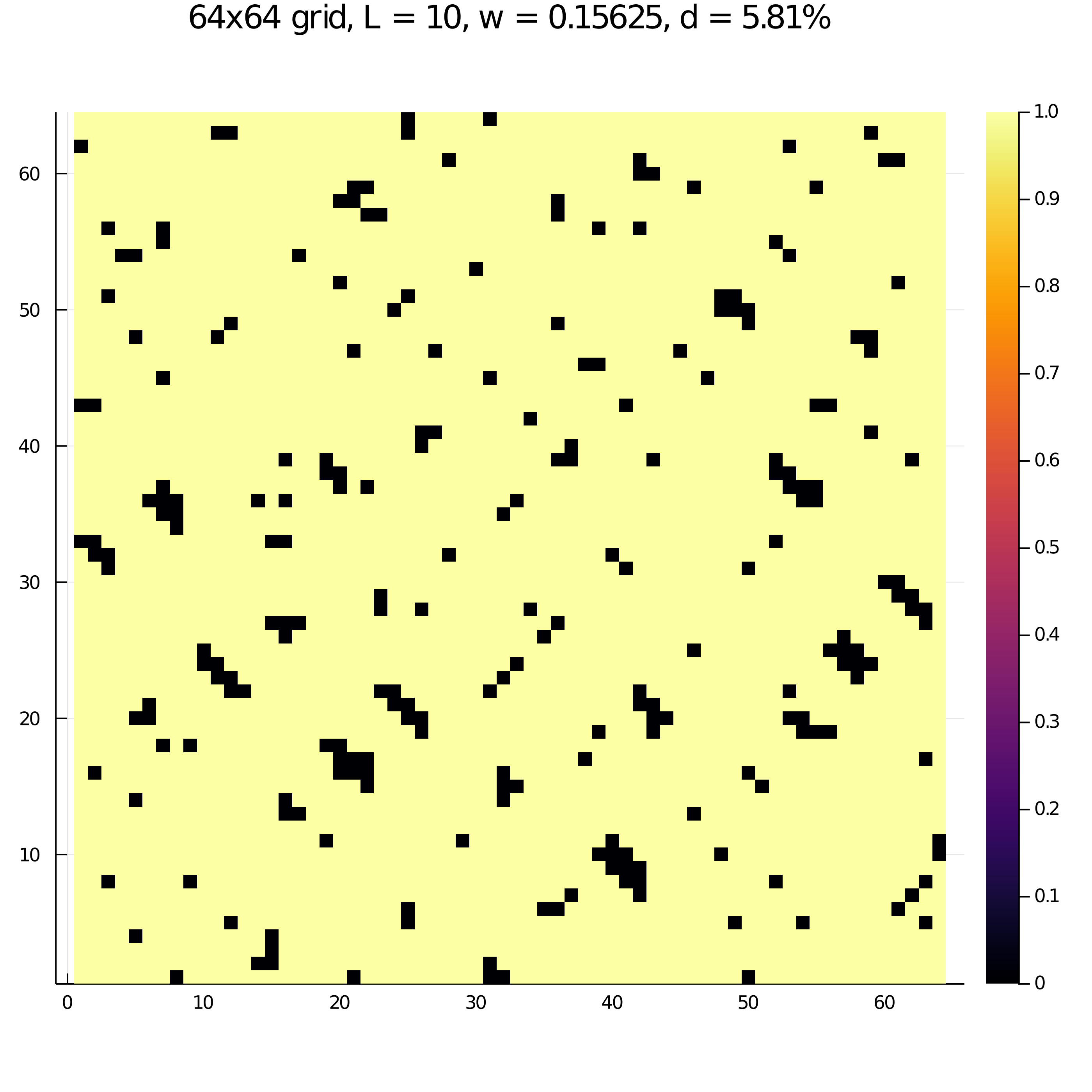} \\
\includegraphics[width=0.3\linewidth]{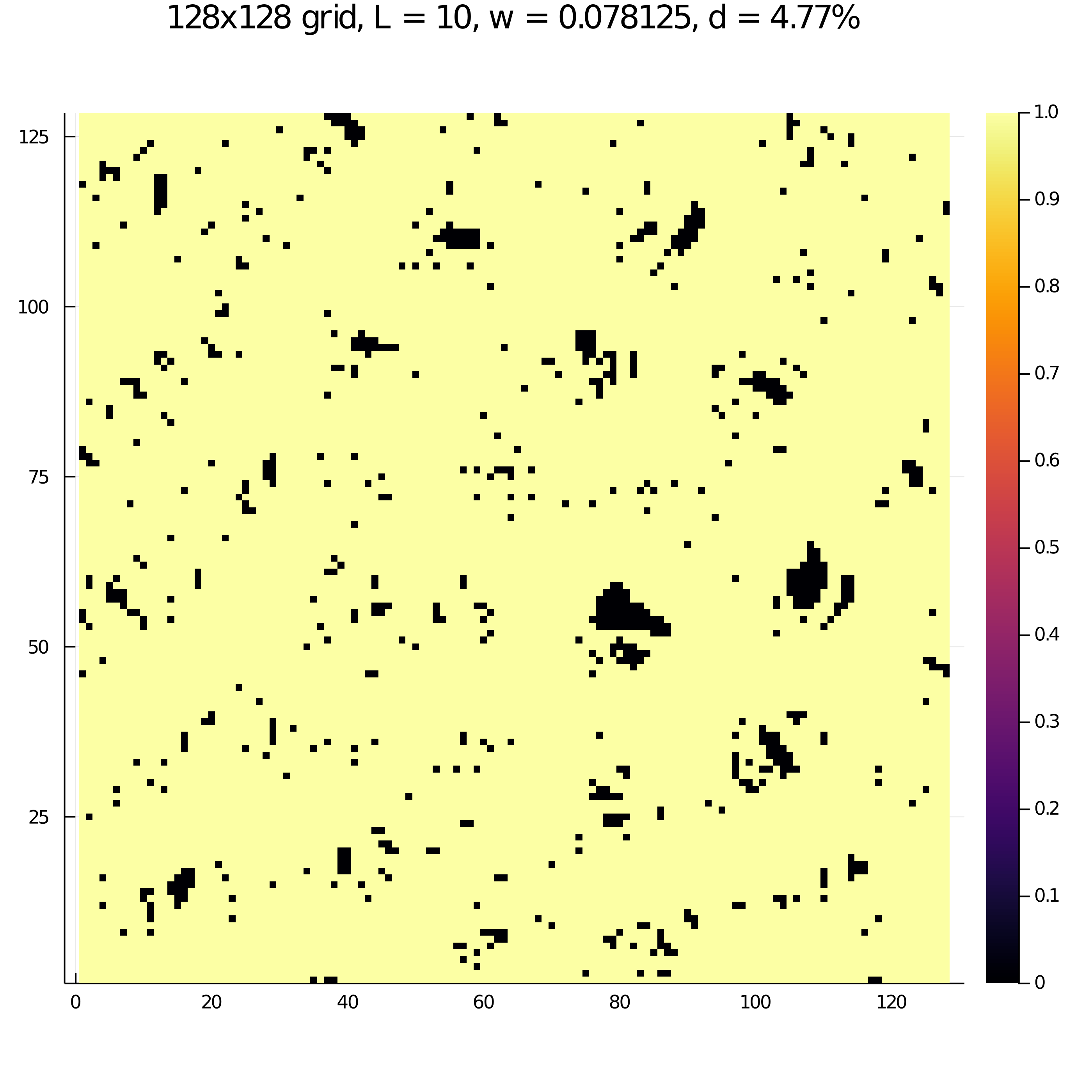} &
\includegraphics[width=0.3\linewidth]{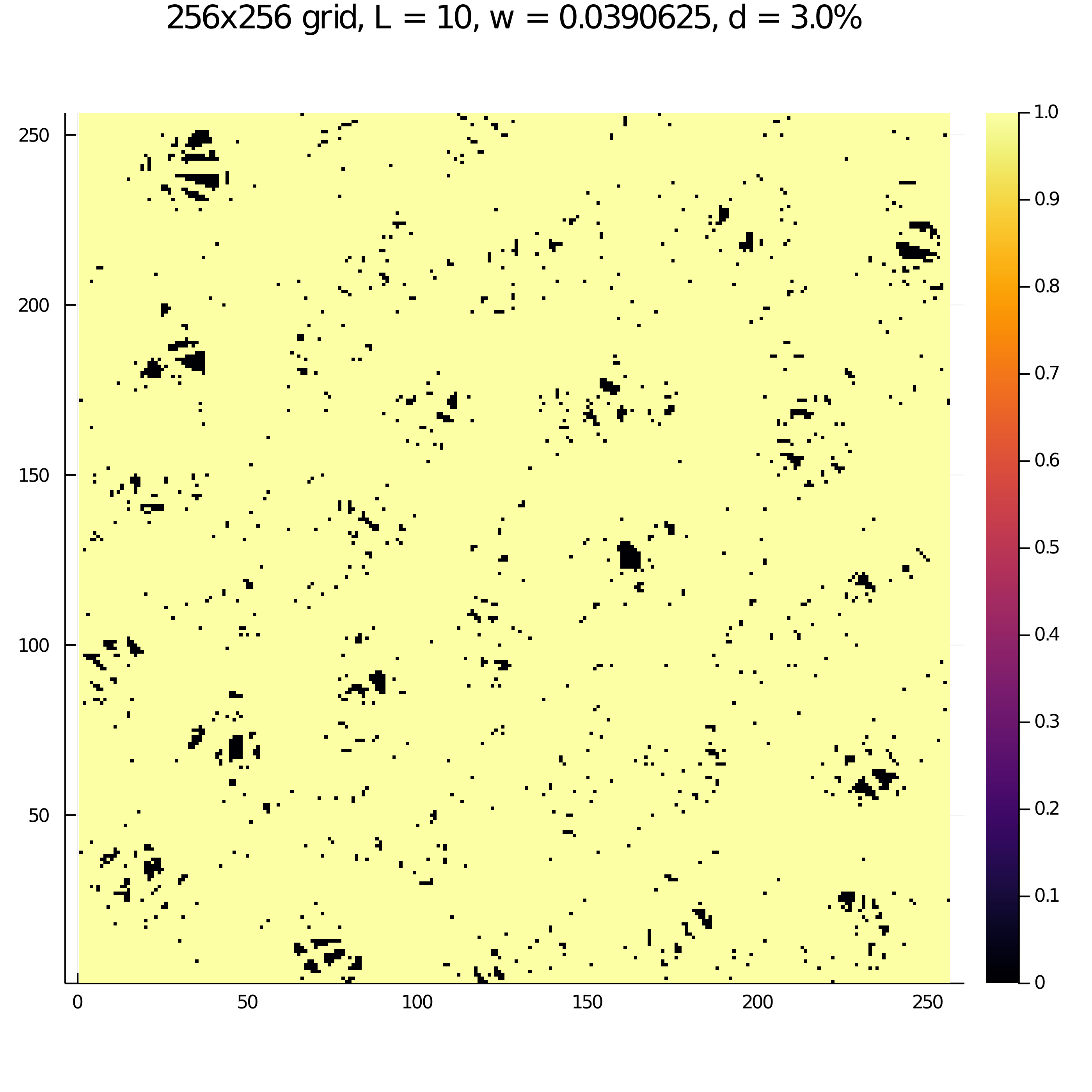} &
\includegraphics[width=0.3\linewidth]{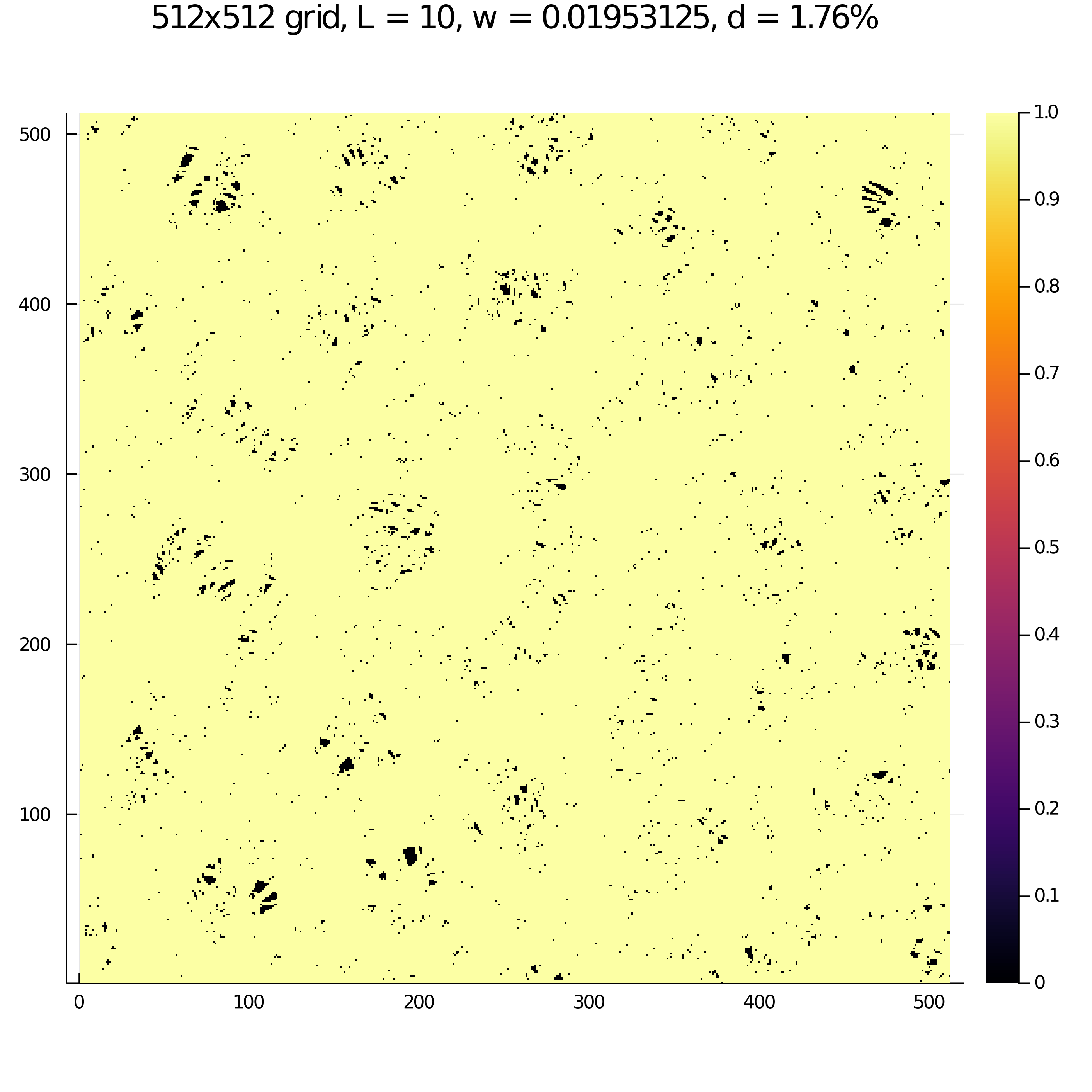} \\
\includegraphics[width=0.3\linewidth]{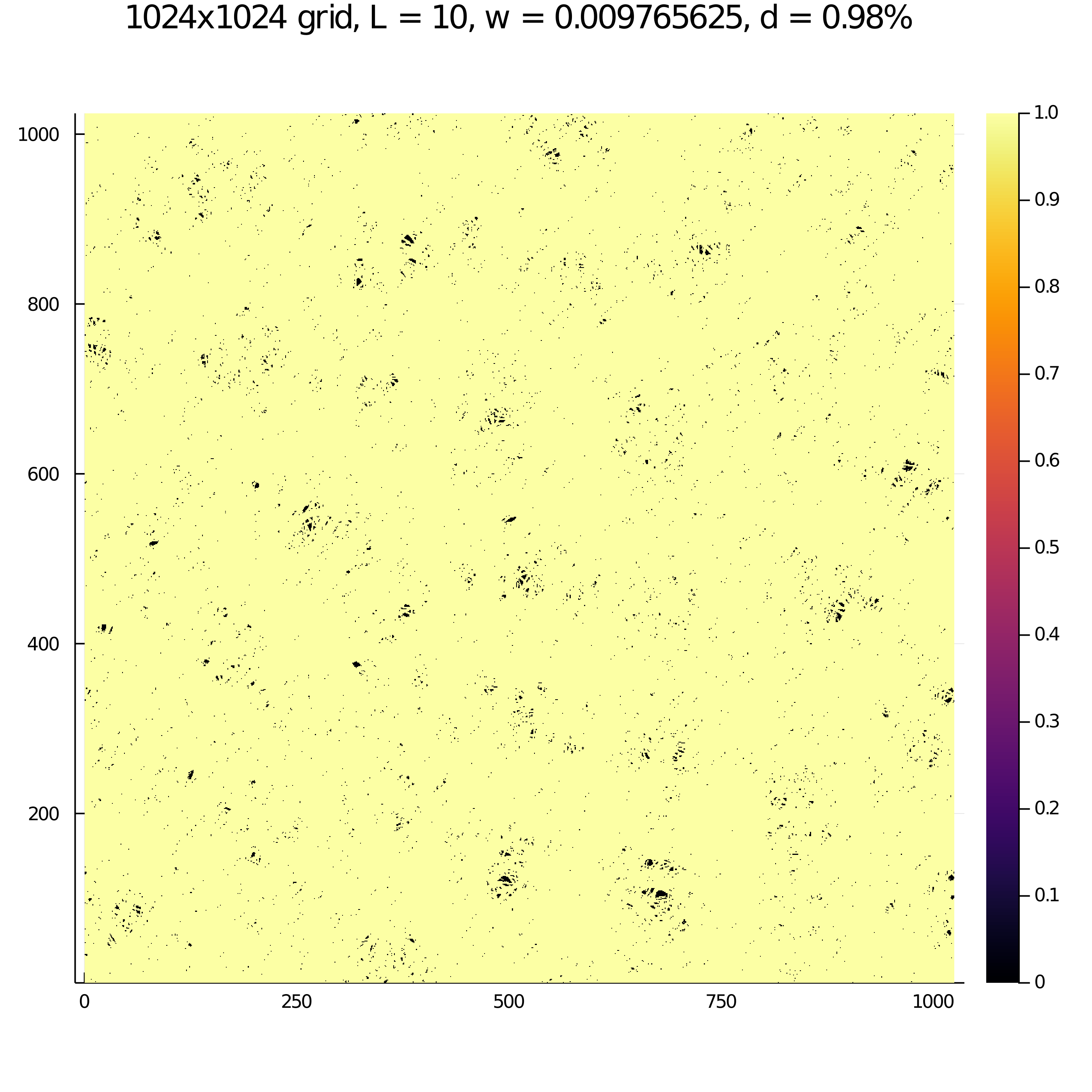} &
\includegraphics[width=0.3\linewidth]{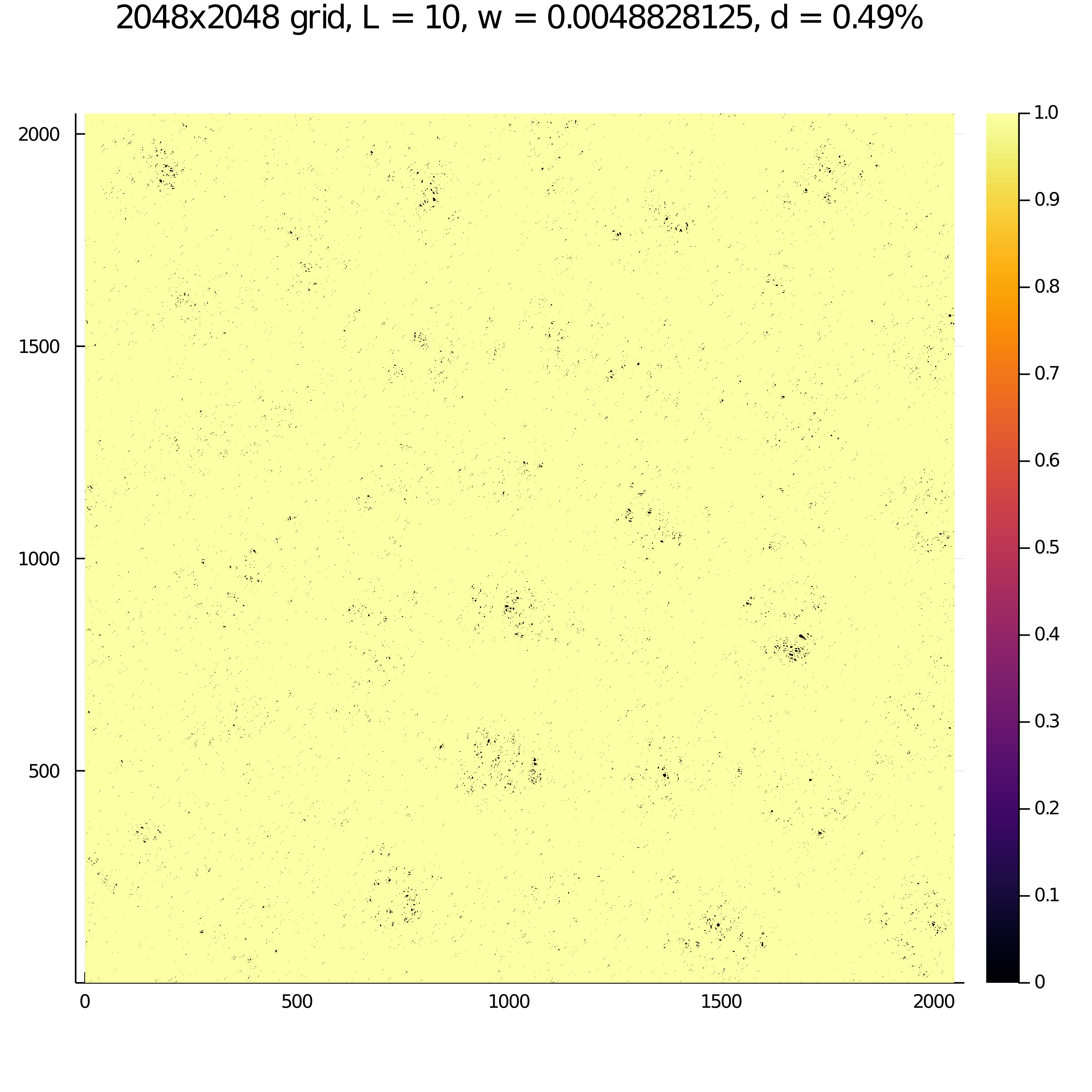} &
\includegraphics[width=0.3\linewidth]{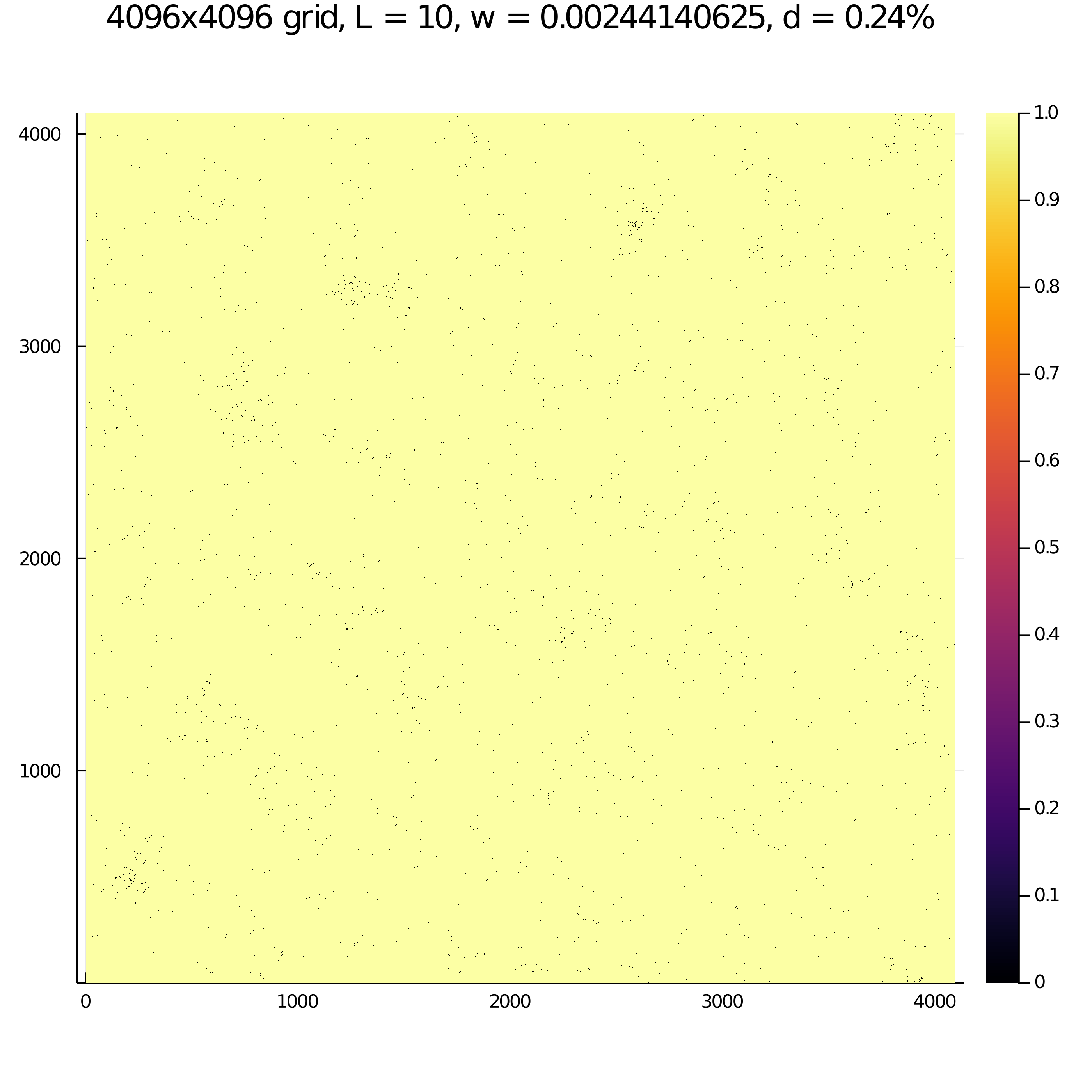} \\
\end{tabular}
\caption{Greedily selected unit distance avoiding sets with varying levels of discretization.} 
\label{tab:greedy_unit_distance_avoiding_sets} 
\end{table}

\begin{figure}[h]
    \centering
    \includegraphics[width=0.75\linewidth]{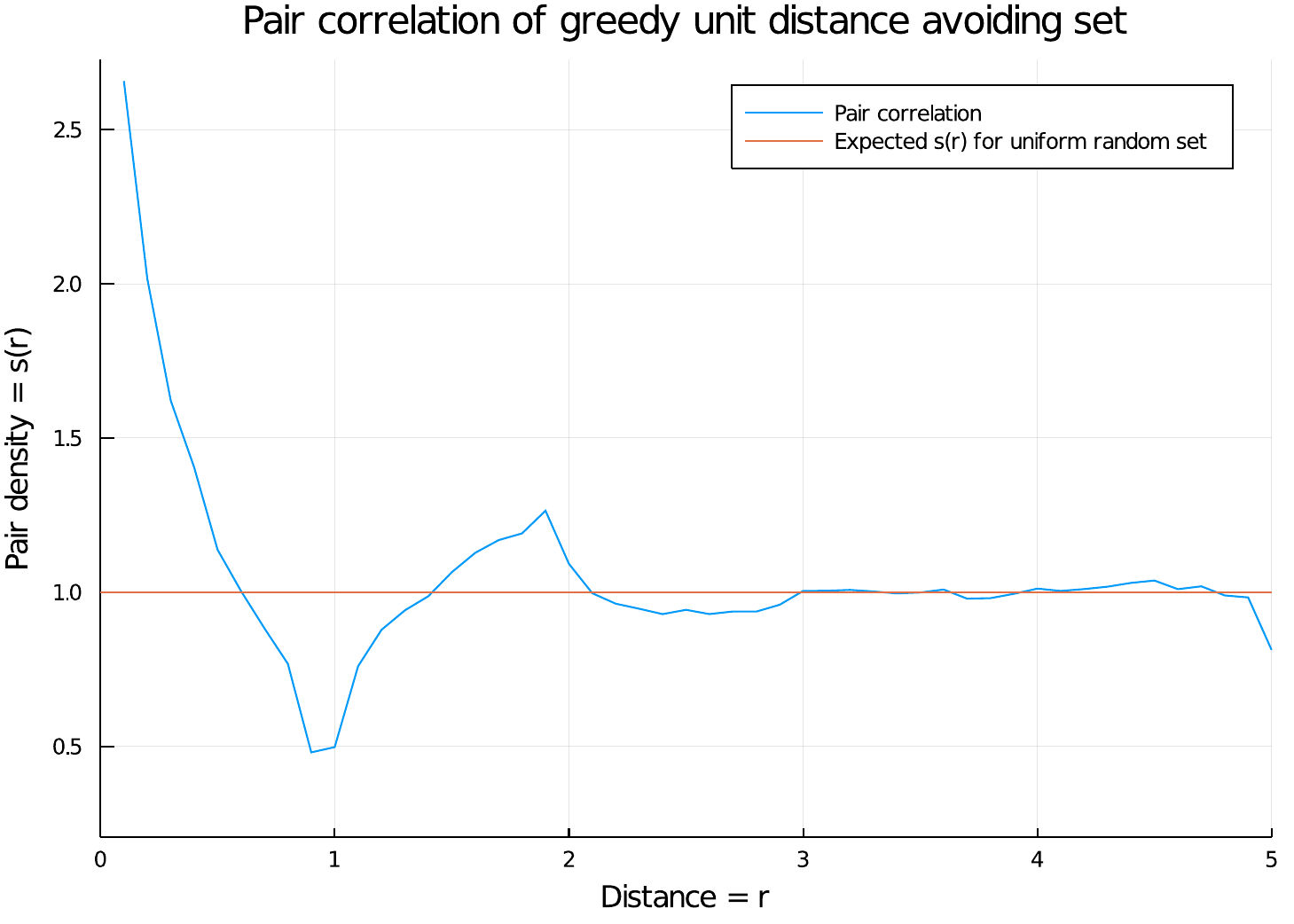}
    \caption{Pair correlation function for greedy unit distance avoiding set with $N = 100$, $K = 10$}
    \label{fig:pair_corr_greedy}
\end{figure}

\clearpage
\bibliographystyle{alpha}
\bibliography{ksat.bib}

\end{document}